\documentclass{article}
 
\title{Perfect sampling algorithms for Schur processes}
\author{D.~Betea\thanks{Laboratoire de Probabilit\'es et Mod\`eles Al\'eatoires, UMR 7599, Universit\'e Pierre et Marie Curie, 4 place Jussieu, F-75005 Paris, \texttt{dan.betea@gmail.com|cedric.boutillier@upmc.fr}} \and
  C.~Boutillier\footnotemark[1] \and
  J.~Bouttier\thanks{Institut de Physique Th\'eorique, Universit\'e Paris-Saclay, CEA, CNRS, F-91191 Gif-sur-Yvette and D\'epartement de Math\'ematiques et Applications, \'Ecole normale sup\'erieure, 45 rue d'Ulm, F-75231 Paris Cedex 05, \texttt{jeremie.bouttier@ipht.fr}} \and
  G.~Chapuy\thanks{LIAFA, CNRS et Universit\'e Paris Diderot, Case 7014, F-75205 Paris Cedex 13, \newline \texttt{(guillaume.chapuy|sylvie.corteel)@liafa.univ-paris-diderot.fr}} \and
  S.~Corteel\footnotemark[3] \and
  M.~Vuleti\'{c}\thanks{Department of Mathematics, University of Massachusetts Boston, Boston, MA 02125, USA, \texttt{mirjana.vuletic@umb.edu}}}

\usepackage[capbesideposition={right,outside},facing=yes,capbesidesep=quad]{floatrow}
\usepackage{graphicx, amsmath, amsthm, amssymb, xcolor}
\usepackage{hyperref}
\usepackage[margin=1.0in]{geometry}
\usepackage{cleveref}
 \usepackage{caption}
\usepackage{subcaption}

\newtheorem{thm}{Theorem}
\newtheorem{prop}[thm]{Proposition}

\theoremstyle{definition}
\newtheorem{definition}[thm]{Definition}
\theoremstyle{remark}
\newtheorem{rem}[thm]{Remark}
\numberwithin{equation}{section}
\numberwithin{thm}{section}

\newcommand{\Z}{\mathbb{Z}}

\newcommand{\Gami}{\Gamma_{-}}
\newcommand{\Gapl}{\Gamma_{+}}
\newcommand{\Gatpl}{\tilde{\Gamma}_{+}}
\newcommand{\Gatmi}{\tilde{\Gamma}_{-}}

\def\fs{\footnotesize}

\begin{document}

\maketitle

\begin{abstract} \small\baselineskip=9pt 

We describe random generation algorithms for a large class of random combinatorial objects called \emph{Schur processes}, which are sequences of random (integer) partitions subject to certain interlacing conditions. This class contains several fundamental combinatorial objects as special cases, such as plane partitions, tilings of Aztec diamonds, pyramid partitions and more generally steep domino tilings of the plane. Our algorithm, which is of polynomial complexity, is both  \emph{exact} (i.e.\ the output follows exactly the target probability law, which is either Boltzmann or uniform in our case), and \emph{entropy optimal} (i.e.\ it reads a minimal number of random bits as an input).

The algorithm encompasses previous growth procedures for special Schur processes related to the primal and dual RSK algorithm, as well as the famous \emph{domino shuffling} algorithm for domino tilings of the Aztec diamond. It can be easily adapted to deal with symmetric Schur processes and general Schur processes involving infinitely many parameters. It is more concrete and easier to implement than Borodin's algorithm, and it is entropy optimal.

At a technical level, it relies on unified bijective proofs of the different types of Cauchy and Littlewood identities for Schur functions, and on an adaptation of Fomin's growth diagram description of the RSK algorithm to that setting. Simulations performed with this algorithm
suggest interesting limit shape phenomena for the corresponding tiling models, some of which are new.
\end{abstract}

\section{Introduction} \label{sec:intro}

Tilings of the plane by pieces of prescribed shapes (such as dominos or rhombi) are fundamental combinatorial objects that have received much attention in discrete mathematics and computer science. In particular, several tiling problems have been studied as models of two dimensional statistical physics, mainly because their remarkable combinatorial structure makes these models physically interesting, algorithmically manageable, and mathematically tractable.

A celebrated example, introduced in~\cite{eklp}, is given by domino tilings of the \emph{Aztec diamond}, see Figure~\ref{fig:aztec}. The first remarkable property of this model is enumerative: the number of domino tilings of the Aztec diamond of size $n$ is $2^{\frac{n(n+1)}{2}}$. This property was proved in~\cite{eklp,eklp2} in several ways, but one of them is of particular importance for the present paper: this result can be proved using a bijective procedure, called the \emph{domino shuffling algorithm}, that generates a tiling of the Aztec diamond of size $n$ taking exactly $\frac{n(n+1)}{2}$ bits as an input. This algorithm is not only elegant but also very useful, since it enabled the efficient sampling of large random tilings, and led to the empirical discovery of the \emph{arctic circle phenomenon} later proved in~\cite{cjp}.

\begin{figure}[!ht]
  \includegraphics[scale=0.40]{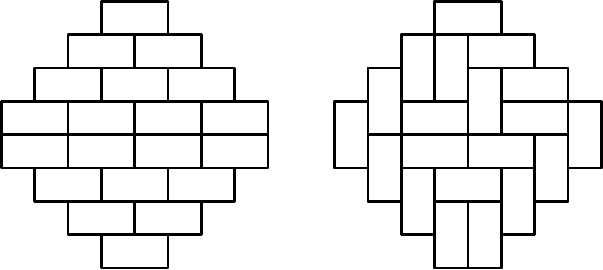}
  \caption{\fs{Two domino tilings of the Aztec diamond of size $4$.}} 
  \label{fig:aztec}
\end{figure}

This paper deals with the extension of this picture to much more general models than that of the Aztec diamond.
Indeed in recent years, many other models of tilings or related objects have been introduced and studied, mainly under the enumerative or ``limit shape" viewpoint. In particular, it was recently observed that tilings of the Aztec diamond are part of a larger family of models of domino tilings of the plane, called \emph{steep tilings}~\cite{bcc}. Steep tilings can themselves be represented as sequences of integer partitions known as \emph{Schur processes} \cite{or, bor}, which puts them under the same roof as other well-known objects such as \emph{plane partitions} \cite{or, or2}. The latter have been much studied as well, and a very elegant and efficient strategy of enumeration and random generation for these objects follows from the Robinson--Schensted--Knuth (RSK) algorithm, see~\cite[Chap.~7]{sta}.  Let us mention that the connection between tilings of the Aztec diamond and Schur processes is implicit in the work of Johansson \cite{joh}. We will also mention that both steep tilings and plane partitions can be seen as dimer matchings in the square and hexagonal lattices respectively. In recent work \cite{bbccr}, it was observed that any Schur process (as defined below) can be realized as a dimer matching on a class of graphs called rail yard graphs. 

Our main result, algorithm {\tt SchurSample} of Section \ref{sec:algo}, is a random generation algorithm for general Schur processes. Remarkably, the algorithm is a common generalization of both the RSK-based strategy for plane partitions and of the domino shuffling algorithm, thus putting those two well-studied combinatorial algorithms under the same roof. It enables one to efficiently sample large random configurations,  which at the  empirical level  unveils some new  properties of their limit shapes. Theorem \ref{thm:main} proves correctness.

At a technical level, our main tools are combinatorial constructions dealing with integer partitions subject to interlacing conditions, that can be viewed as bijective proofs of the different  Cauchy identities for Schur functions. The obtained algorithm takes as an input a finite sequence of geometric random variables and random bits that can be represented in a graphical way in terms of \emph{growth diagrams} similar to Fomin's description of the RSK algorithm, see e.g.~\cite[Appendix]{sta}. In this setting, it is easy to see that the algorithm is entropy optimal (Proposition~\ref{prop:entropy}). The (polynomial) complexity is also studied (Proposition~\ref{prop:comp}).

We then adapt our algorithm in order to produce samples from \emph{symmetric Schur processes}. These processes are defined on symmetric sequences of partitions, or in a different interpretation on free boundary sequences. They are related to free boundary or symmetric tilings of the plane. Examples include symmetric plane partitions and plane overpartitions. We produce the {\tt SymmetricSchurSample} algorithm and two variants in Section~\ref{sec:symm}. The modifications we make to obtain these algorithms are related to the Littlewood identities for Schur functions and allow us to obtain generalizations of the symmetric RSK algorithm. Other adaptations of our algorithm are discussed in Sections~\ref{sec:unbounded} and \ref{sec:gensample}, and correspond to limiting cases such as unboxed plane partitions or Plancherel-type measures.

\smallskip
We conclude this introduction with references to previous works, and a discussion on how this paper relates to them. The idea of growing Schur processes dynamically has been floating around even before these objects were invented, in the context of the RSK correspondence, and was the subject of several papers. Gessel \cite{ges}, Krattenthaler \cite{kra}, Pak--Postnikov \cite{pp}, and Fomin~\cite[Chap.7]{sta} all discuss RSK-type bijections that are based on a growing procedure, and could be used for sampling special kinds of Schur processes corresponding to the primal-RSK and dual-RSK situation (and indeed in the primal-case of Section \ref{sec:algo}, the basic building-block we are using is Gessel's bijection). However, those works do not contain a growing procedure for these objects that both covers the mixed primal/dual case \emph{and} that is totally concrete and bijective. The first motivation of this paper is to fill this gap, and this level of concreteness enables us to transform the growing 
procedure into an effective sampling algorithm, that we have entirely implemented.

In a different context, Borodin~\cite{bor} gives a random sampling algorithm for fully general Schur processes. It is also based on a growing procedure and similarly uses geometric random variables as an input, but contrarily to ours it is not entropy optimal (although this question is not discussed in~\cite{bor}, it is easy to see by comparing with the present paper). Our algorithm, which is based only on two very explicit simple growing rules, is arguably easier to implement. Its simplicity and its entropic efficiency are the second motivation for our work. Let us also note that growing dynamics have been described for generalizations of Schur processes such as Macdonald processes, see e.g.\ the recent paper of Borodin and Petrov \cite{bp}, at the cost of an increase in complexity and entropy (and atomic steps are no longer bijections). Finally, RSK-type dynamics have also been a useful tool in the analysis of various probabilistic models \cite{gtw,oco}.

A last situation (maybe more surprisingly related to ours) where growing procedures were employed arises in the study of the \emph{Aztec diamond}. As mentioned above, Elkies--Kuperberg--Larsen--Propp's \emph{domino shuffling algorithm} \cite{eklp2} is a procedure that enables one to recursively grow uniform tilings of an Aztec diamond of increasing size. Initially invented as a combinatorial device with miraculous properties, the Aztec diamond is now understood to be part of a large class of domino tilings of the plane~\cite{bcc} \emph{in bijection} with certain Schur processes. As it turns out, the domino shuffling algorithm can be viewed as a special case of our growth procedure, thus revealing a part of the structure hidden behind this algorithm (another part is given by the connection with cluster algebras \cite{spe}). The third motivation of this paper is therefore to put the domino shuffling under the same roof as growing procedures for Schur processes and the RSK correspondences.  
\smallskip

To be complete, let us finally mention a few previous works related to the sampling of tilings of the plane that are not directly related to ours. Propp and Wilson's \emph{coupling from the past} method~\cite{pw} is often used for random sampling of domino tilings. In a different direction, the reader interested in exhaustive sampling of tilings (outside the realm of Schur processes) could consult \cite{dr} and references therein.
 
\smallskip

The paper is organized as follows. In Section \ref{sec:schur} we remind the definition of Schur processes and explain its relation to tilings through examples. We also describe the necessary prerequisites on partitions, Schur functions and the vertex operator formalism needed in the rest of the paper.  In Section \ref{sec:algo} we give the main sampling algorithm for Schur processes after initially describing the two main bijections we use in said algorithm. We also give some samples of large tilings obtained using these algorithms. In Section \ref{sec:symm} we modify the algorithm to suit symmetric (free boundary) Schur processes and provide some samples obtained using the algorithms. We discuss sampling from the so-called unbounded Schur process in Section~\ref{sec:unbounded} and from the most general Schur process in Section~\ref{sec:gensample}. We conclude in Section \ref{sec:conclusion}.

\section{Reminders on Schur processes} \label{sec:schur}

\subsection{Basic definitions} 

A \emph{partition} $\lambda$ is a nonincreasing sequence of nonnegative integers $\lambda_1 \geq \lambda_2 \geq \cdots $ that vanishes eventually. We call each positive $\lambda_i$ a \emph{part} and the number of parts, called the \emph{length} of $\lambda$, is denoted $\ell(\lambda)$.  The \emph{empty partition} $\emptyset$ is the partition of length zero. We call $|\lambda|:=\sum_{i=1}^{\ell(\lambda)} \lambda_i$ the \emph{weight} of the partition.  For any $\lambda$ we have a \emph{conjugate partition} $\lambda'$ whose parts are defined as $\lambda'_i := |\{j: \lambda_j \geq i\}|$. It is often convenient to represent partitions graphically  by either Young or Maya diagrams, see Figure~\ref{maya}.  A \emph{Maya diagram} is an encoding of a partition $\lambda$  as a boolean function $\Z+\frac{1}{2} \to \{ \circ, \bullet \}$, where particles $\bullet$ are assigned to the half-integers of the form $\lambda_i-i+1/2+n$, $i=1,2,\dots$ and holes $\circ$ to all the others, for some integer $n$.  Starting with a Maya diagram one recovers parts of the partition by counting  the number of holes to the left of each particle. The positions of the holes actually encode the conjugate partition, being precisely the half-integers of the form $-\lambda'_i+i+1/2+n$, $i=1,2,\dots$.

\begin{figure}[!ht]
  \centering
  \begin{subfigure}[b]{0.2\textwidth}
  \centering
  \includegraphics{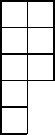}
\caption{}
\end{subfigure}
  \begin{subfigure}[b]{0.2\textwidth}
  \centering
  \includegraphics{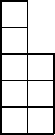}
\caption{}
  \end{subfigure}
  \begin{subfigure}[b]{0.3\textwidth}
  \centering
  \includegraphics{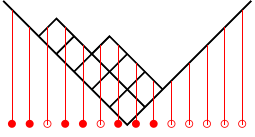}
\caption{}
  \end{subfigure}
  \caption{\fs{The Young diagram of the partition $(2,2,2,1,1)$ shown in: (a) English convention, (b) French convention and (c) Russian convention and a Maya diagram associated to it.}}
  \label{maya}
\end{figure}

Let $\lambda$ and $\mu$ be two partitions with $\lambda \supseteq \mu$ (that is, $\lambda_i \geq \mu_i$ for all $i$). They are said to be \emph{interlacing} and we write $\lambda \succ \mu$ if and only if $\lambda_1 \geq \mu_1 \geq \lambda_2 \geq \mu_2 \geq \lambda_3 \geq \cdots$. The interlacing property is equivalent to saying that the skew diagram $\lambda/\mu$ (the set-difference of the Young diagrams of $\lambda$ and $\mu$) is  a \emph{horizontal strip}, i.e.\ it has at most one square in each column. There is a notion of \emph{dual interlacing}: we write $\lambda \succ' \mu$ if $\lambda' \succ \mu'$, or equivalently if the skew diagram $\lambda/\mu$ is a \emph{vertical strip}, i.e.\  $0 \leq \lambda_i-\mu_i \leq 1$ for all $i$. 

For a word $w=(w_1,w_2,\dots,w_n) \in \{\prec,\succ,\prec',\succ'\}^n$, we say that a sequence of partitions $\Lambda = (\emptyset=\lambda(0), \lambda(1), \dots, \lambda(n) = \emptyset)$ is \emph{$w$-interlaced} if $\lambda(i-1) \, w_i \,\lambda(i)$, for $i=1, \dots, n$. Also, word multiplication is defined as concatenation, e.g.\ 
$(\prec)^3(\prec',\succ)^2=(\prec,\prec,\prec,\prec',\succ,\prec',\succ)$. 
We are now ready to define Schur processes.

\begin{definition} \label{def:schurdef}
For a word $w=(w_1,w_2,\dots,w_n) \in \{\prec,\succ,\prec',\succ'\}^n$, the \emph{Schur process} of word $w$ with parameters $Z = (z_1, \dots, z_n)$ is the measure on the set of $w$-interlaced sequences of partitions $\Lambda = (\emptyset=\lambda(0), \lambda(1), \dots, \lambda(n) = \emptyset)$ given by 
\begin{align} \label{eq:schur-measure}
 Prob(\Lambda) \propto \prod_{i=1}^n z_i^{||\lambda(i)|-|\lambda(i-1)||}.
\end{align}
\end{definition}

\begin{rem}
  There is actually a more general definition of Schur processes
  \cite{or, bor} which we discuss in Section~\ref{sec:unbounded}. The
  definition we use for the moment covers most applications, and is
  actually the same one for which other exact sampling algorithms were
  given in \cite[Section~7]{bor}.
\end{rem}

The most natural specialization of the measure~\eqref{eq:schur-measure} (which also ensures convergence) is obtained by  choosing a parameter $0 < q < 1$ and then choosing the $z_i$ parameters such that 
\begin{equation*}
Prob(\Lambda) \propto q^{\text{Volume}(\Lambda)} = q^{\sum_{i} |\lambda(i)|}.
\end{equation*}
This can be accomplished by choosing 
\begin{equation}
  z_i = \left\{
    \begin{array}{ll}
      q^{-i} & \text{if }  w_i \in \{\prec,\prec'\}\\
      q^i & \text{if } w_i \in \{\succ, \succ'\}\\
    \end{array}.
  \right.
  \label{eq:qspec}
\end{equation}

For some concrete words there might be other, more convenient, choices of $z_i$ that also give the $q^\text{Volume}$ weight. 

Before we proceed with examples, we introduce an \emph{encoded shape} associated with $w=(w_1,w_2,\dots,w_n) \in\{\prec,\succ,\prec',\succ'\}^n$.  Construct a path consisting of horizontal and vertical  unit length segments, choosing a horizontal segment  if $w_i \in  \{\prec,\prec'\},$   and vertical if $w_i \in \{\succ, \succ'\}$, and moving in the right-downward direction,  as in Figure~\ref{schur-proc}. This path can be seen as the boundary of a Young diagram (in French convention) which we call  the encoded shape and denote with $sh(w)$. It is not hard to see that parts of the partition corresponding to $sh(w)$ are obtained by counting the number of $\{\prec,\prec'\},$ elements of $w$ lying to the left of a fixed $w_i \in \{\succ, \succ'\}$.

\begin{figure}[!ht]
\begin{center}
      \includegraphics{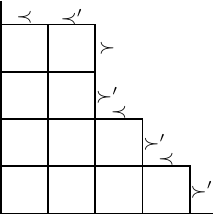}
\end{center}
\caption{\fs{The encoded shape of $w=(\prec,\prec',\succ,\succ',\prec,\succ',\prec,\succ')$ is $sh(w)=(4,3,2,2)$.}}
\label{schur-proc} 
\end{figure}

\subsection{Examples}

There are several special cases of interest, listed below. They all correspond to various types of tilings. 

\paragraph{Reverse plane partitions} of shape $\mathcal{S}$, where $\mathcal{S}$ is a Young diagram in French convention, are fillings of $\mathcal{S}$ with nonnegative integers that form nondecreasing rows and columns, see Figure~\ref{tilings-ex}. In other words, if squares in $\mathcal{S}$ are represented with their position $(i,j)$ ($i$ being the row and $j$ the column the square belongs to, where the bottom-left square is represented with (1,1)), and $\pi_{i,j}$ is the filling of  $(i,j)$, then we have $\pi_{i,j}\leq \pi_{k,l}$ whenever $i \leq k$ and $j\leq l$.  

For $w\in \{\prec,\succ\}^n$, it is not hard to see that  $w$-interlaced sequences are in correspondence with reverse plane partitions of shape $sh(w)$, also know as \emph{skew plane partitions}  when zero fillings are ignored. The corresponding interlacing partitions are diagonal slices of the reverse plane partition. 
  
In particular, reverse plane partitions of shape $m^n=(\overbrace{m,\dots,m}^n)$, which are also known as $(m \times n)$-boxed \emph{plane partitions},  correspond to  $w$-interlaced sequence where $w=(\prec)^m(\succ)^n$.

\begin{figure}[ht]
  \includegraphics{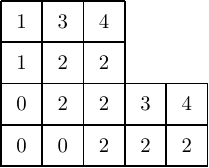} \quad \quad \quad \includegraphics[scale=0.5]{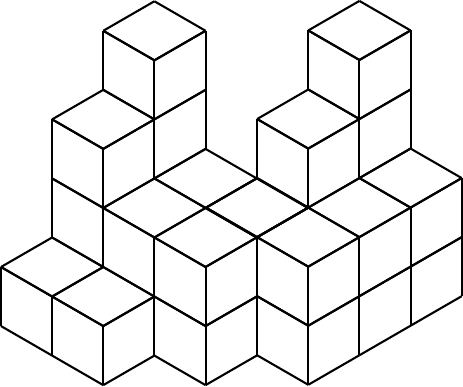} 
\caption{\fs{A reverse plane partition of shape (5,5,3,3) corresponding to the interlacing sequence $\emptyset \prec (1) \prec (3,1) \prec (4,2) \succ (2,2) \succ (2) \prec (3,2) \prec (4,2) \succ (2) \succ \emptyset$.}}  
\label{tilings-ex}
\end{figure}

\paragraph{Domino tilings of the Aztec diamond} of size $n$
 are in correspondence with $w$-interlaced sequences where $w =(\prec',\succ)^n$. 
The correspondence is described in the part about steep tilings, see Figure~\ref{AztecDiamondPyramidPartitions}(a) for illustration of the correspondence. The encoded shape is a staircase partition, i.e.\ $(n,n-1,\dots,1)$. Notice that since the number of tilings is bounded,  there are no convergence issues and one can get any general Schur distribution with generic $Z$ parameters.

\paragraph{Pyramid partitions}\cite{you} are infinite heaps of blocks of size $2 \times 2 \times 1$ stacked in a way that resembles pyramids. For a precise definition we first start with a minimal pyramid partition that is an infinite heap of blocks shown in Figure~\ref{PyramidPartitions}(b).  For convenience, we use two different colors to represent odd layers versus even layers.  A pyramid partition is obtained by removing a finite number of blocks from the minimal  pyramid partition where a block can be removed if there are no blocks on top of it, see Figure~\ref{PyramidPartitions}(c) for an example.

\begin{figure}[ht]
\centering
\begin{subfigure}[b]{0.2\textwidth}
\centering
\includegraphics[scale=0.4]{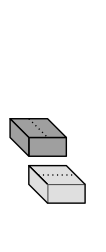}
\caption{}
\end{subfigure}
\begin{subfigure}[b]{0.35\textwidth}
\centering
\includegraphics[scale=0.4]{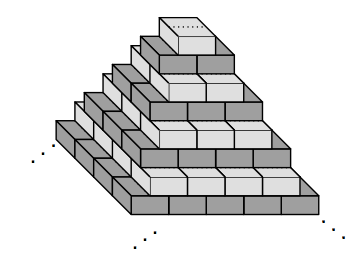}
\caption{}
\end{subfigure}
\begin{subfigure}[b]{0.43\textwidth}
\centering
\includegraphics[scale=0.4]{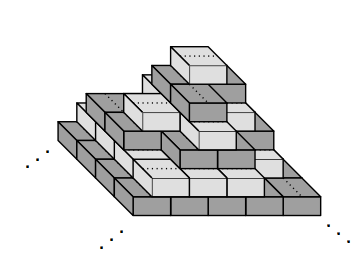}
\caption{}
\end{subfigure}
\caption{\fs{Building blocks (a), a minimal pyramid partition (b), a pyramid partition (c).}} 
\label{PyramidPartitions}
\end{figure}

A view from the top of a pyramid partition reveals a domino tiling corresponding to the pyramid partition, see Figure~\ref{Pyramid PartitionsTilings}. 
\begin{figure}[ht]
\includegraphics[scale=0.45]{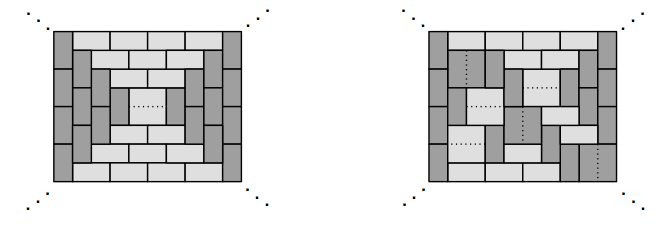}
\caption{\fs{Domino tilings corresponding to the pyramid partitions from Figure~\ref{PyramidPartitions}. }} 
\label{Pyramid PartitionsTilings}
\end{figure}

A pyramid partition of width $l$ is obtained from the minimal pyramid partition by removing blocks that lie inside the strip $-l \leq x - y \leq l$, where the center of the top block is placed at (0,0), see Figure~\ref{Pyramid PartitionsTilings}. 

It can be shown that pyramid partitions of width $l$ correspond to $w$-interlaced sequences where $w=(\underbrace{\dots, \prec,\prec',\prec,\prec'}_l,\underbrace{\succ,\succ',\succ,\succ',\dots}_l)$. See Figure \ref{AztecDiamondPyramidPartitions}(b) for an illustration of the correspondence which is explained below, in the part about steep tilings.

\begin{figure}[ht]
  \begin{subfigure}[b]{0.4\textwidth}
  \centering
  \includegraphics{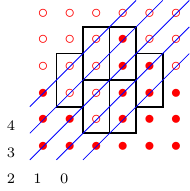}
\caption{}
\end{subfigure}
   \begin{subfigure}[b]{0.4\textwidth}
   \centering
   \includegraphics{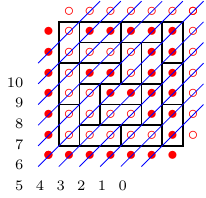}
\caption{}
\end{subfigure}
\caption{\fs{(a) A $2\times 2$ Aztec diamond corresponding to the sequence $\emptyset  \prec' (1,1) \succ (1) \prec' (2)  \succ \emptyset $. (b) A pyramid partition of width 5 corresponding to the sequence $\emptyset \prec' (1) \prec (1,1) \prec' (2,2) \prec (2,2,2) \prec' (3,3,2) \succ (3,2) \succ' (2,1) \succ (2) \succ' (1) \succ \emptyset$.}} 
\label{AztecDiamondPyramidPartitions}
\end{figure}

\paragraph{Steep tilings} generalize both Aztec diamonds and pyramid partitions, and are roughly speaking domino tilings of a strip. Before we proceed with a precise definition we introduce some conventions. Assume we are given a domino tiling of a square grid and that the grid is colored in a chessboard fashion. A domino is called positive if its top-left corner belongs to a white square, otherwise the domino is called negative.  Also, if a square belongs to a positive domino then to its center we assign a hole $\circ$, otherwise we assign a particle $\bullet$. 

Let a word $w=(w_1,\dots,w_{2l}) \in \{\prec,\succ,\prec',\succ'\}^{2l}$ be such that $w_{2i}\in\{\prec,\succ\}$ and $w_{2i+1}\in \{\prec',\succ'\}$. Construct a path of length $2l$ consisting of horizontal and vertical  unit length segments  for each $i=1,\dots, 2l$ by choosing a horizontal segment  if
\begin{equation*}
  w_i = \left\{
  \begin{array}{ll}
    \prec' & \text{if }  i \text{ is odd},\\
    \succ & \text{if } i \text{ is even} \\
  \end{array} 
  \right.
\end{equation*}
and vertical otherwise, and moving in the right-downward direction,  as in Figure~\ref{MinimalTiling}. Then add an infinite strip of dominoes like in Figure~\ref{MinimalTiling}. We call this domino tiling a $w$-minimal tiling. Up and to the right of the path are positive dominoes and down and to the left are negative dominoes. 
\begin{figure}[ht]
  \includegraphics{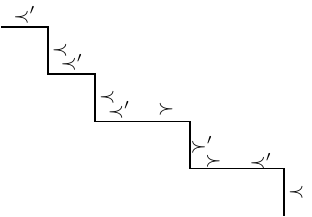} \quad \includegraphics[scale=0.5]{AsymptoticData.pdf} 
\caption{\fs{A $w$-minimal tiling for $w=(\prec',\prec,\prec',\prec,\prec',\succ,\succ',\succ,\prec',\prec)$.} 
\label{MinimalTiling}}
\end{figure}

A $w$-steep tiling is a domino tiling obtained from the $w$-minimal tiling by performing finitely many flips, i.e.\ by replacing a pair of two adjacent vertical dominoes with a pair of two adjacent horizontal dominoes or vice versa.  One such example is shown in Figure~\ref{SteepTilings}. If we assign particles and holes as described above, we get a representation of a steep tiling as a sequence of partitions (where each diagonal slice is a partition represented by its Maya diagram), see Figure~\ref{SteepTilings}. 
\begin{figure}[ht]
 \includegraphics[scale=0.5]{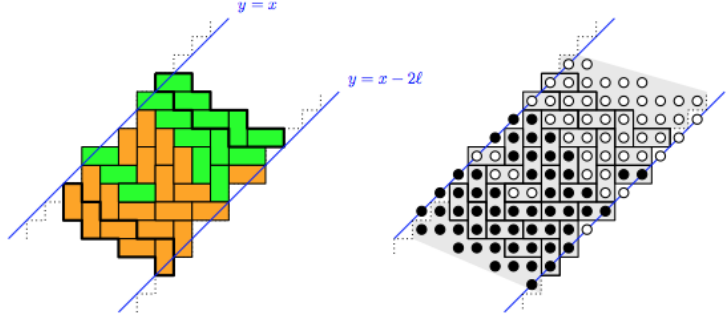} 
 \caption{\fs{A $w$-steep tiling for $w=(\prec',\prec,\prec',\prec,\prec',\succ,\succ',\succ,\prec',\prec).$ }}
 \label{SteepTilings}.
\end{figure}

Let $w=(w_1,\dots,w_{2l})$ be such that $w_{2i}\in\{\prec,\succ\}$ and $w_{2i+1}\in \{\prec',\succ'\}$. Using the correspondence shown in Figure~\ref{SteepTilings}, it can be shown that  $w$-step tilings are in bijection with $w$-interlaced sequences,  see~\cite{bcc} for details of the proof. In particular, $w=(\prec',\succ)^n$ corresponds to $n \times n$ Aztec diamonds and $w=(\prec',\prec)^l(\succ',\succ)^l$ to pyramid partitions of width $2l$.

\subsection{Schur processes via symmetric functions and vertex operators} 

Schur processes can be defined using tools coming from representation
theory or theoretical physics, namely Schur symmetric functions and
vertex operators. We present this here.

We start with the definition of Schur functions, which can be defined in many ways, for  example using the Jacobi--Trudi formula, see e.g.~\cite{mac},
\begin{equation} \label{eq:jacobitrudi}
 s_{\lambda} (x_1,\dots, x_n) = \det_{1 \leq i,j \leq n} h_{\lambda_i - i + j}(x_1,\dots,x_n).
 \end{equation}
Here the $h$'s are the complete symmetric functions. By convention $s_{\emptyset}(\cdot) = 1$ and $s_{\lambda}() = \delta_{\lambda, \emptyset}$. We only give the definition for finitely many variables as this is enough for our purposes.  Related are the skew Schur functions, which for two partitions $\lambda, \mu$ are zero unless $\mu \subseteq \lambda$, in which case they are defined by
\begin{equation} \label{eq:jacobitrudiskew}
s_{\lambda / \mu} (x_1, \dots, x_n) = \det_{1 \leq i,j \leq n} h_{\lambda_i - \mu_j - i + j}(x_1, \dots ,x_n).
\end{equation}
Note that $s_{\lambda/\emptyset} = s_{\lambda}$. Let us denote the alphabet $(x_1, \dots, x_n)$ by $X$ (and similarly for $Y$). For our purposes, the important identities satisfied by Schur functions are the branching rule
\begin{align} \label{branching}
s_{\lambda / \mu} (X,Y) = \sum_{\nu} s_{\lambda / \nu}(X) s_{\nu / \mu}(Y)
\end{align}
\noindent and the Cauchy (and dual Cauchy) identities:
\begin{align}
  \sum_{\nu} s_{\nu / \lambda} (X) s_{\nu / \mu} (Y) &= \prod_{i,j} \frac{1}{1-x_i y_j} \sum_{\kappa} s_{\lambda / \kappa} (Y) s_{\mu  / \kappa} (X), \label{cauchy}\\
  \sum_{\nu} s_{\nu / \lambda} (X) s_{\nu' / \mu'} (Y) &= \prod_{i,j} (1+x_i y_j) \sum_{\kappa} s_{\lambda' / \kappa'} (Y) s_{\mu  / \kappa} (X). \label{dual-cauchy}
\end{align}

Schur functions can also be defined as a generating series of semi-standard Young tableaux. Precisely, for $\mu \subseteq \lambda$, 
\begin{equation}
  s_{\lambda/ \mu} (x_1, \dots, x_n) = \sum_{T} x_i^{\# \ \text{of}\ i \ \text{in}\ T},\label{eq:SchurSYT}
\end{equation}
where the sum ranges over all semi-standard Young tableaux of shape $\lambda / \mu$ (fillings of the skew diagram $\lambda / \mu$ with numbers 1 through $n$ such that the numbers weakly increase in rows, from left to right, and strictly increase down columns). Note that, in one variable, we have
\begin{equation}
  \label{eq:Schur1V}
  s_{\lambda / \mu} (x_1) = x_1^{|\lambda| -|\mu|} \delta_{\lambda \succ \mu}.
\end{equation}
The fact that the two definitions are equivalent is the content of the Lindstr\"{o}m--Gessel--Viennot lemma, see e.g.~\cite{sta}. The Cauchy identities can be proven from the tableaux definition using the RSK correspondence~\cite{knu}.

In the infinite wedge formalism (we refer the reader to, e.g., \cite{or, or2} for details), to each partition $\lambda$ one associates a basis vector  $| \lambda \rangle$ (respectively covector $\langle \lambda |$) in the half-infinite wedge vector space denoted in the literature by $\bigwedge^{\frac{\infty}{2} } V$ (respectively in the dual vector space). We add a bilinear form defined by $\langle \lambda | \mu \rangle = \delta_{\lambda, \mu}$. Two important operators, each depending on a parameter, acting on these vectors are $\Gapl(z)$ and $\Gami(z)$.  We call them \emph{vertex operators}. Operator $\Gapl(z)$ removes, in all possible ways, a horizontal strip with weight $z^{\# \ \text{boxes}}$ from a partition. Operator $\Gami(z)$ adds a horizontal strip. More precisely, we have
\begin{equation*}
\Gapl(z) |\lambda \rangle := \sum_{\mu \prec \lambda}  z^{|\lambda|-|\mu|}|\mu\rangle, \qquad \Gami(z) |\lambda \rangle := \sum_{\mu \succ \lambda} z^{|\mu|-|\lambda|} |\mu\rangle.
\end{equation*}

\noindent We also define  operators $\Gatpl(z)$ and  $\Gatmi(z)$ which respectively remove or add vertical strips, i.e.
\begin{align*}
\Gatpl(z) |\lambda \rangle := \sum_{\mu \prec' \lambda} z^{|\lambda|-|\mu|}|\mu\rangle, \qquad \Gatmi(z) |\lambda \rangle := \sum_{\mu \succ' \lambda} z^{|\mu|-|\lambda|}  |\mu\rangle.
\end{align*}

Using the tableaux definition, or the branching rule,  one can show that Schur functions take the following form in terms of the $\Gamma$ operators:
\begin{align}
s_{\lambda / \mu}(x_1, \dots ,x_n) & = \langle \mu | \Gapl(x_1) \cdots \Gapl(x_n) | \lambda \rangle = \langle \lambda | \Gami(x_1) \cdots \Gami(x_n) | \mu \rangle, \label{SchurGa}\\
s_{\lambda' / \mu'}(x_1, \dots, x_n) & = \langle \mu | \Gatpl(x_1) \cdots \Gatpl(x_n) | \lambda \rangle = \langle \lambda | \Gatmi(x_1) \cdots \Gatmi(x_n) | \mu \rangle.\label{DualSchurGa}
\end{align}

Observe the action of these operators on the vacuum vector (empty partition vector): 
\begin{align} \label{vacuum-action}
 \Gapl(x) |\emptyset \rangle = \Gatpl(x) |\emptyset \rangle = |\emptyset \rangle, \ \langle \emptyset| \Gami(y) = \langle \emptyset| \Gatmi(y) = \langle \emptyset|.
\end{align}
The following commutation relations hold: 
\begin{align}
  \Gapl(x) \Gami(y) &= \frac{1}{1-xy} \Gami(y) \Gapl(x), \label{comm-hh}\\
  \Gatpl(x) \Gatmi(y) &= \frac{1}{1-xy} \Gatmi(y) \Gatpl(x), \label{comm-vv}\\
  \Gatpl(x) \Gami(y) &= (1+xy) \Gami(y) \Gatpl(x), \label{comm-vh}\\
  \Gapl(x) \Gatmi(y) &= (1+xy) \Gatmi(y) \Gapl(x). \label{comm-hv}
\end{align}
They also satisfy the trivial commutation relations whereby any two operators whose indices are the same (both $+$ or both $-$) commute, regardless of parameters.

It is not hard to see that the commutation relations \eqref{comm-hh} and \eqref{comm-vh} are equivalent to the Cauchy \eqref{cauchy} and dual Cauchy \eqref{dual-cauchy} identities. Obviously this is true when $X=(x)$ and $Y=(y)$, and in general this is true by the branching rule, \cref{SchurGa,DualSchurGa}. The other two commutation relations are also equivalent to the Cauchy identities by just conjugating all the partitions involved. In Section \ref{sec:algo} we will describe bijections that allow us to give (bijective) proofs of the commutation relations, which we further use to obtain our exact sampling algorithm.

The Schur process of word $w$ with parameters $Z$ can be written as
\begin{equation}
  Prob(\Lambda)\propto\prod_{i=1}^n\langle \lambda(i-1) |\Gamma_i(z_i)|\lambda(i)\rangle
  \label{eq:SPvert}
\end{equation}
where $\Gamma_i$ is $\Gamma_+$, $\Gamma_-$, $\tilde{\Gamma}_+$, $\tilde{\Gamma}_-$ if $w_i$ is $\prec,\succ,\prec',\succ'$, respectively. The partition function is given by
\begin{equation}
  Z_w = \sum_{\Lambda}\prod_{i=1}^n z_i^{||\lambda(i)|-|\lambda(i-1)||}= \langle \emptyset | \prod_{i=1}^n\Gamma_i(z_i) | \emptyset \rangle.
  \label{eq:SPZ}
\end{equation}

A simple application of the above commutation relations,  along with the action of the vertex operators on the vacuum vectors~\eqref{vacuum-action} yields:
\begin{prop} \label{prop:SPZexpr}
  The partition function of the Schur process of word $w$ with parameters $Z$ is equal to
\begin{equation}
  Z_w = \prod_{i<j, w_i \in \{\prec, \prec'\}, w_j \in \{\succ, \succ'\} } (1+ \epsilon_{i,j} z_i z_j)^{\epsilon_{i,j}}, 
 \label{eq:SPZexpr}
\end{equation}
where $\epsilon_{i,j} = 1$ if $(w_i, w_j) \in \{ (\prec, \succ'), (\prec', \succ)\}$ and $\epsilon_{i,j} = -1$ otherwise.
\end{prop}

\section{Bijective sampling of Schur processes} \label{sec:algo}

Exact sampling algorithms based on partition processes for classes of tilings including plane and skew plane partitions \cite{bor} (see also \cite{bf} for a general approach and \cite{bf2} for an application of this approach to a restricted class of graphs including, but not limited to, the Aztec diamond graph) and Aztec diamonds (see \cite{eklp2} and remark below) have been proposed before. There is also the coupling from the past approach of Propp and Wilson \cite{pw}. In this section we give an overall encompassing algorithm for Schur processes (different from, but similar to, that of \cite{bor}) which under appropriate specializations gives exact random sampling of plane (and skew plane) partitions, Aztec diamonds, pyramid partitions and more generally, steep tilings. 

Our algorithm is based on bijective proofs of the Cauchy identities for Schur functions, which, as we have explained earlier,  are equivalent to the commutation relations between $\Gapl(x)$ and $\Gami(y)$, or, in the dual case, between $\Gapl(x)$ and $\Gatmi(y)$. Bijections are represented  schematically in Figure~\ref{atomic-steps}, and the corresponding commutation relations on which they are based in Figure~\ref{comm-rel-diagram}. They are given in simple terms and lead to an algorithm that is easy to implement. We note that bijective proofs of Cauchy identities and versions of bijections we describe here appeared earlier in the literature. For the work related to the Cauchy case see Gessel \cite{ges}, and work related to the dual case see Pak--Postnikov \cite{pp} and Krattenthaler \cite{kra}.  Recent work of Borodin--Petrov \cite{bp} puts the Cauchy case in the more general framework of nearest neighbor dynamics. In particular, in Section 7 they provide two bijections for the Cauchy case (one listed below), both of which could be used for a sampling algorithm.

\begin{figure}[ht]
    \centering
    \begin{subfigure}[b]{0.45\textwidth}
    \centering
    \includegraphics{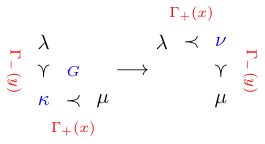}
    \caption{}
    \end{subfigure}
    \qquad
    \begin{subfigure}[b]{0.45\textwidth}
    \centering
    \includegraphics{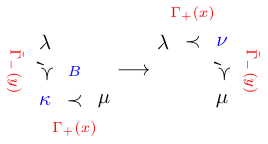}
    \caption{}
    \end{subfigure}
    \caption{\fs{A diagrammatic representation of the two bijections. In
        the Cauchy case (a) we map a pair $(\kappa,G)$ such that $\lambda
        \succ \kappa \prec \mu$ and $G \in \mathbb{N}$ to $\nu$ such that
        $\lambda \prec \nu \succ \mu$. In the dual Cauchy case (b) we map
        a pair $(\kappa,B)$ such that $\lambda \succ' \kappa \prec \mu$
        and $B \in \{ 0, 1\}$ to $\nu$ such that $\lambda \prec \nu \succ'
        \mu$.}}
    \label{atomic-steps}
\end{figure}

\begin{figure}[ht]
\centering
\begin{subfigure}[b]{0.45\textwidth}
\centering
\includegraphics{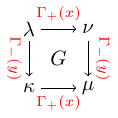}
\caption{}
\end{subfigure}
\quad
\begin{subfigure}[b]{0.45\textwidth}
\centering
\includegraphics{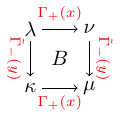}
\caption{}
\end{subfigure}
 \caption{\fs{The atomic commutation relations of $\Gapl(x)$ and $\Gami(y)$ (a), and of $\Gapl(x)$ and $\Gatmi(y)$ (b); $G \in \mathbb{N}, B \in \{ 0, 1\}$.}}
 \label{comm-rel-diagram}
\end{figure}

\paragraph{Cauchy case} Let $\lambda, \kappa, \mu$ be three partitions
such that $\lambda \succ \kappa, \mu \succ \kappa$ and let $G \in
\mathbb{N}$. We describe a procedure for building a fourth partition $\nu$
with the properties that $\lambda \prec \nu, \mu \prec \nu$ and
$|\lambda| + |\mu| + G = |\kappa|+|\nu|$, in such a way that the
mapping $(\kappa,G) \mapsto \nu$ is bijective (i.e.\ every possible
$\nu$ is obtained once and exactly once by the procedure). See
Figure~\ref{atomic-steps}-a for a schematic representation. We
construct $\nu$ (which has at most $\max(\ell(\lambda),\ell(\mu))+1$
parts) by setting
\begin{equation} \label{bij-HH}
  \nu_i =
  \begin{cases}
    \max(\lambda_1, \mu_1) + G & \text{if $i=1$}, \\
    \max(\lambda_i, \mu_i) + \min(\lambda_{i-1}, \mu_{i-1}) -
    \kappa_{i-1} & \text{if $i>1$}\\
  \end{cases}
\end{equation}
and it is readily checked that the wanted properties hold.

Using the bijection we deduce
\begin{equation}
  \sum_{\substack{\nu \\ \nu \succ \lambda \\ \nu \succ \mu}} x^{|\nu|-|\lambda|} y^{|\nu|-|\mu|} =\sum_{G \geq 0} \sum_{\substack{\kappa \\ \kappa \prec \lambda \\\kappa \prec \mu}}x^{|\mu|-|\kappa|+G} y^{|\lambda|-|\kappa|+G}= \frac{1}{1-xy} \sum_{\substack{\kappa \\ \kappa \prec \lambda \\\kappa \prec \mu}} x^{|\mu|-|\kappa|} y^{|\lambda|-|\kappa|}.
  \label{eq:ggcomm}
\end{equation}
This identity amounts to  the commutation relation \eqref{comm-hh}, and to the Cauchy identity \eqref{cauchy} in the case where $X$ and $Y$ are reduced to a single variable. 

To obtain a bijective proof of the commutation relation \eqref{comm-vv}, one applies the same procedure after a priori conjugating all the partitions and conjugating the resulting $\nu$ at the end. 

The first bijective procedure (proving~\eqref{comm-hh}) has time complexity $O(\ell(\nu))$ while the second (proving~\eqref{comm-vv}) is $O(\max(\ell(\nu), \nu_1))$ so both are commonly $O(\max(\ell(\nu), \nu_1))$ (and we will need this weaker bound in order to give uniform estimates below).

The above bijection can be turned into a random sampling procedure easily. With the notation set up above, define a method {\tt sampleHH} (HH stands for horizontal-horizontal and the fact that we are commuting $\Gapl$ and $\Gami$) that does the following:

\indent {\fs \textbf{def} {\tt sampleHH}($\lambda, \mu, \kappa,\xi$) \\
\indent \indent sample $G \sim Geom(\xi)$ \\
\indent \indent construct $\nu$ based on the bijective procedure described above \\
\indent \indent \textbf{return} $\nu$}

\noindent where $Geom(\xi)$ samples a single geometric random variable from the distribution $Prob(k) \propto \xi^k$ (we assume $0 \leq \xi < 1$). Then applying this procedure for $\xi=xy$ will produce $\nu$ distributed as $Prob(\nu) \propto s_{\nu / \lambda} (x) s_{\nu / \mu} (y)$ using minimal entropy (the sampling of a single geometric random variable) and assuming $\lambda, \nu, \kappa$ are coming from a Schur process, in the end so will $\nu$ (from a different Schur process, of course). 

A method {\tt sampleVV} can be defined in the same way and corresponds to commuting $\Gatpl$ and $\Gatmi$.

\smallskip

\paragraph{Dual Cauchy case} Let $\lambda, \kappa, \mu$ be three
partitions such that $\kappa \prec' \lambda, \mu \succ \kappa$ and let
$B \in \{0,1\}$. We describe a procedure similar to the above one for
building a fourth partition $\nu$ with the properties that $\lambda
\prec \nu, \mu \prec' \nu$ and $|\lambda| + |\mu| + B =
|\kappa|+|\nu|$, in such a way that the mapping $(\kappa,B) \mapsto
\nu$ is bijective. See Figure~\ref{atomic-steps}-b for a
schematic representation. Here we directly define the method {\tt sampleHV}
(HV stands for horizontal-vertical and the fact that we are commuting
$\Gapl$ and $\Gatmi$):

\indent{\fs \textbf{def} {\tt sampleHV}($\lambda, \mu, \kappa,\xi$) \\
  \indent \indent sample $B \sim Bernoulli(\frac{\xi}{1+\xi})$ \\
  \indent \indent \textbf{for} $i=1 \dots \max(\ell(\lambda),\ell(\mu))+1$ \\
  \indent \indent \indent \textbf{if} $\lambda_i \leq \mu_i < \lambda_{i-1}$ \textbf{then} $\nu_i=\max(\lambda_i,\mu_i)+B$\\
  \indent \indent \indent \textbf{else}  $\nu_i=\max(\lambda_i,\mu_i)$\\
  \indent \indent \indent \textbf{if} $\mu_{i+1} < \lambda_i \leq \mu_i$ \textbf{then} $B=\min(\lambda_i,\mu_i)-\kappa_i$\\
  \indent \indent \textbf{return} $\nu$}

\noindent where $Bernoulli(\xi)$ is a Bernoulli random variable that
returns 0 with probability $1-\xi$ and 1 with probability $\xi$. To
check the validity of our method, note first that the interlacing
conditions for $\kappa$ and $\nu$ amount to
\begin{equation*}
\max(\lambda_i-1, \mu_{i+1}) \leq \kappa_i \leq \min(\lambda_{i}, \mu_i), \quad \max(\lambda_i, \mu_i) \leq \nu_i \leq \min(\lambda_{i-1}, \mu_i+1)
\end{equation*}
where by convention $\lambda_0 = \infty$. In particular, the quantity
$\min(\lambda_{i}, \mu_i)-\kappa_i$ vanishes unless
\begin{equation}
  \label{eq:kapcond}
  \mu_{i+1} < \lambda_i \leq \mu_i
\end{equation}
in which case it may also take the value $1$. Let $i_1<i_2<\cdots<i_r$
be the $i$'s such that \eqref{eq:kapcond} holds. Similarly, the
quantity $\nu_i-\max(\lambda_{i}, \mu_i)$ vanishes unless
\begin{equation}
  \label{eq:nucond}
  \lambda_i \leq \mu_i < \lambda_{i-1}
\end{equation}
in which case it may also take the value $1$. Let
$j_1<j_2<\cdots<j_s$ be the $i$'s such that \eqref{eq:nucond}
holds. Our method works provided that $s=r+1$ and
\begin{equation} \label{eq:blockinterlace}
  j_1 \leq i_1 < j_2 \leq i_2 < \cdots <j_r \leq i_r < j_{r+1},
\end{equation}
which follows from the easily checked fact that the mapping
\begin{equation*}
i \mapsto \min\{j > i, \lambda_j \leq \mu_j\}
\end{equation*}
defines a bijection between $\{0,i_1,i_2,\ldots,i_r\}$ and
$\{j_1,j_2,\ldots,j_s\}$. The $(\kappa,B)\mapsto \nu$ bijection implies 
\begin{equation}
  \sum_{\substack{\nu \\ \nu \succ \lambda \\ \mu \prec' \nu}} x^{|\nu|-|\lambda|} y^{|\nu|-|\mu|} = \sum_{B \in \{0,1\}} \sum_{\substack{\kappa\\ \kappa \prec' \lambda \\ \kappa \prec \mu}}x^{|\mu|-|\kappa|+B} y^{|\lambda|-|\kappa|+B}=(1+xy) \sum_{\substack{\kappa\\ \kappa \prec' \lambda \\ \kappa \prec \mu}} x^{|\mu|-|\kappa|} y^{|\lambda|-|\kappa|},
  \label{eq:ggpcomm}
\end{equation}
which amounts to the commutation relation \eqref{comm-hv}, and to the dual Cauchy identity \eqref{dual-cauchy} in the case where $X$ and $Y$
are reduced to a single variable. 

\begin{figure}[thb]
  \centering
  \includegraphics[scale=0.5]{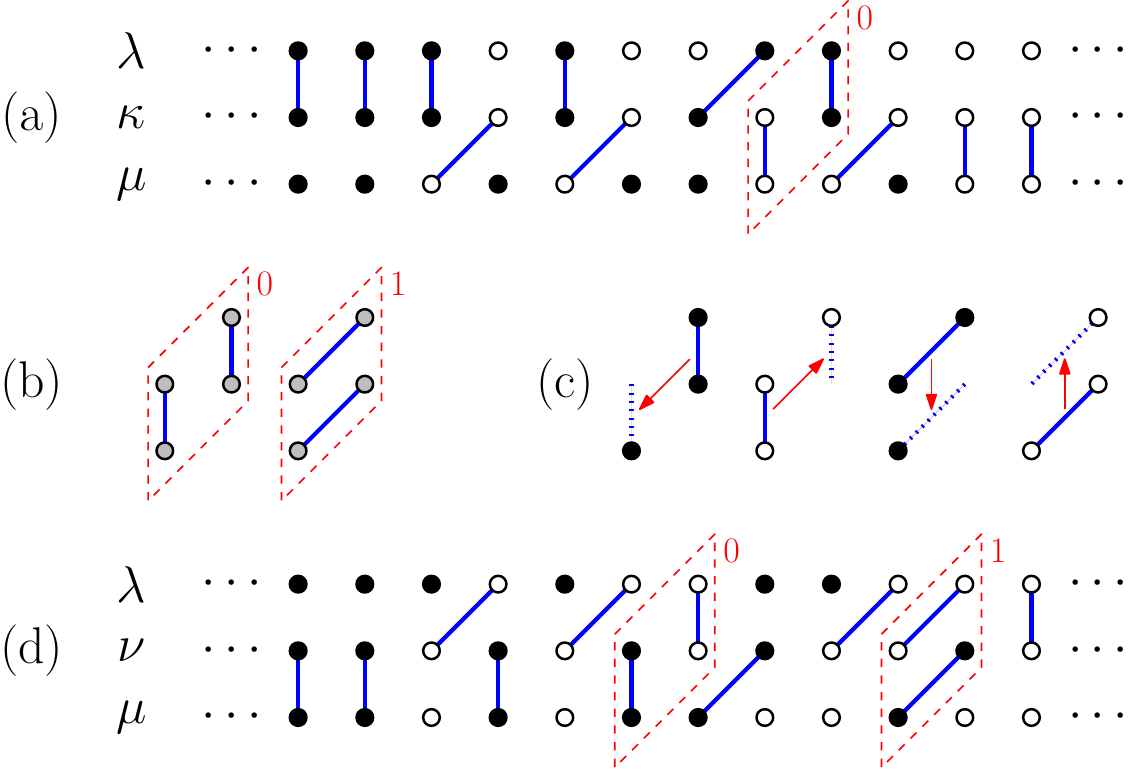}
  \caption{\fs{Dimer shuffling interpretation of the bijection used in the
    dual Cauchy case. Given $\lambda,\kappa,\mu$ such that
    $\lambda \succ' \kappa \prec \mu$, we represent their Maya
    diagrams on top of each other (a). The interlacing conditions
    ensure that we may match the particles $\bullet$ of the Maya
    diagrams of $\lambda$ and $\kappa$ together, and the holes $\circ$
    of those of $\kappa$ and $\mu$ together, in order to form dimers
    (blue) that are either vertical or diagonal ($45^\circ$). Removing
    the ``blocks'' (b) and recording the corresponding bits (red), we
    slide the remaining dimers according to the rules (c), which are
    nothing but those of \cite[Figure~14]{eklp2} in disguise. Using
    the recorded bits and the extra bit $B$ (here equal to $1$) to
    fill the blocks, we obtain bijectively another matching (d) with
    reversed convention, where the middle Maya diagram corresponds to
    $\nu$ such that $\lambda \prec \nu \succ' \mu$ and
    $|\nu|+|\kappa|=|\lambda|+|\mu|+B$.}}
  \label{fig:shuffling}
\end{figure}

\begin{rem}
  The above bijection admits an interpretation, explained on
  Figure~\ref{fig:shuffling}, in terms of matchings (dimers) which
  will be useful later for the identification with the domino
  shuffling algorithm. Note that the positions of the ``blocks''
  (Figure~\ref{fig:shuffling}-b) are precisely given by the $i$'s
  satisfying \eqref{eq:kapcond} and \eqref{eq:nucond} respectively, so
  that $\eqref{eq:blockinterlace}$ ensures that the blocks of
  $(\lambda,\kappa,\mu)$ and $(\lambda,\nu,\mu)$ are interlaced
  together.
\end{rem}

The complexity of {\tt sampleHV} is $O(\ell(\nu))$  which is bounded
above by $O(\max(\ell(\nu), \nu_1))$ (recall we will need this weaker bound to produce uniform estimates below). Applying the procedure for $\xi=xy$ will
produce $\nu$ distributed as $Prob(\nu) \propto s_{\nu / \lambda} (x)
s_{\nu' / \mu'} (y)$ using minimal entropy (the sampling of a single
Bernoulli random variable) and assuming $\lambda, \nu, \kappa$ are
coming from a Schur process, in the end so will $\nu$ (from a
different Schur process, of course).

A method {\tt sampleVH}, of the same complexity, is defined by exchanging the roles of $\lambda$ and
$\mu$ and corresponds to commuting $\Gatpl$ and $\Gami$ according to~\eqref{comm-vh}.

We introduce a type which can be one of HH, VV, HV or VH and wrap the four atomic sample steps described above in a single method which we call sample:

\indent {\fs \textbf{def} sample($\lambda, \mu, \kappa, \xi, type$) \\
\indent\indent \textbf{case} type: \\
\indent\indent\indent HH: \textbf{return} sampleHH($\lambda, \mu, \kappa, \xi$) \\
\indent\indent\indent HV: \textbf{return} sampleHV($\lambda, \mu, \kappa, \xi$) \\
\indent\indent\indent VH: \textbf{return} sampleVH($\lambda, \mu, \kappa, \xi$) \\
\indent\indent\indent VV: \textbf{return} sampleVV($\lambda, \mu, \kappa, \xi$)}

We now give the exact sampling algorithm for the Schur process of word $w$ with parameters $Z$. We first do a pre-computation step, which produces ${\tt{Par}}(w,Z)=(\pi,X,Y,\text{getType})$, in a way that we describe below.  Suppose that $w$ has $m$ elements in $\{\prec,\prec'\}$ and $n$ elements in $\{\succ,\succ'\}$. We set $\pi$ to be the partition corresponding to the encoded shape $sh(w)$, see Figure~\ref{schur-proc}. Obviously, $m=\pi_1$ and $n=\ell(\pi)$. Let $i_1\leq \dots  \leq i_m$ be the indices of elements of $w$ in $\{\prec,\prec'\}$, and  $j_n \leq  \dots \leq j_1$ the indices  of the others. Set  $u_k=w_{i_k}$ and $x_k=z_{i_k}$, for $1 \leq k \leq m$ and $v_k=w_{j_k}$ and $y_k=z_{j_k}$, for $1 \leq k \leq n$.  Then, set $X=(x_1,\dots,x_m)$ and $Y=(y_1,\dots,y_n)$. We also build a function getType which for each $(i,j) \in \pi$ returns  HH, VV, HV or VH if $(u_i,v_j)$ is $(\prec,\succ),(\prec',\succ'),(\prec,\succ')$ or $(\prec',\succ)$, respectively. For the example from 
Figure~\ref{schur-proc} we have $\pi=(4,3,2,2)$, $X=(z_1,z_2,z_5,z_7)$, $Y=(z_8,z_6,z_4,z_3)$, $(u_1,\dots,u_4)=(\prec,\prec',\prec,\prec)$, $(v_1,\dots,v_4)=(\succ',\succ',\succ',\succ)$  and getType function values are shown in Figure~\ref{GetType}.

\begin{figure}[!ht]
\begin{center}
    \includegraphics{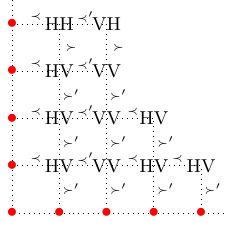}
\end{center}
\caption{\fs{Values of getType function for $w=(\prec,\prec',\succ,\succ',\prec,\succ',\prec,\succ')$.}}
\label{GetType} 
\end{figure}

The idea of the algorithm is to build the shape $\pi$, one square at a time, starting from the empty partitions on the coordinate axes (red bullets in Figure \ref{GetType}), where at each square a partition is produced according to the type of the square.
\newline

\indent {\fs \textbf{Algorithm} {\tt SchurSample}} \\
\indent {\fs \textbf{Input:} $\pi$, partitions $\tau(0,0) = \tau(0,i) = \tau(j,0) = \emptyset$ and parameters $x_i, y_j$, $1 \le i \leq m, 1\leq j \le n$\\
\indent \textbf{for} $j=1 \dots \ell(\pi)$ \\
\indent\indent \textbf{for} $i=1 \dots \pi_j$ \\
\indent\indent\indent $type$ = getType($i,j$) \\
\indent\indent\indent $\tau(i,j)$=sample($\tau(i-1,j),\tau(i,j-1),\tau(i-1,j-1), x_i y_j, type$) \\
\indent \textbf{Output:} The sequence of partitions $\Lambda$ defined by \[\Lambda := (\emptyset = \tau(l_0), \tau(l_1),\dots,\tau(l_{m+n-1}),\tau(l_{m+n})=\emptyset),\]
\noindent where $(l_0 = (0,n), l_1 ,\dots, l_{m+n-1}, l_{m+n} = (m,0))$ is the \emph{ordered} sequence of lattice points on the boundary of $\pi$ clockwise from the vertical to the horizontal axis.
}

Our main result is then 
\begin{thm}\label{thm:main}
  The algorithm {\tt SchurSample} produces an exact random sample from
  the Schur process corresponding to the word $w$ and the parameter
  list $Z$. This remains true if, in the algorithm, we replace the
  double \textbf{for} loop by any loop which runs over all the boxes
  $(i,j)$ of the encoded shape $\pi=sh(w)$ respecting their partial
  order (i.e.\ such that $\tau(i_1,j_1)$ is sampled before
  $\tau(i_2,j_2)$ if $i_2 \geq i_1$ and $j_2 \geq j_1$).
\end{thm}

\begin{proof}
Let $w$ be a word. Denote with $m$ the number of elements of $w$ in $\{\prec,\prec'\}$ and with $n$ in $\{\prec,\prec'\}$. We prove the statement is true for fixed $m$, $n$ and fixed parameter list $(X,Y)$ (obtained from $Z$ as described above). The proof is by induction on $|sh(w)|$. If $|sh(w)| = 0$ then all $\{\succ,\succ'\}$ come before $\{\prec,\prec'\}$ and so the Schur process is supported on a single element,  the sequence of empty partitions, which we sample trivially.

Suppose we can exactly sample a Schur process of $w$  for any word $w$ such that $|sh(w)|\leq k$. Let $w$ be a word with $|sh(w)|=k+1$. Then $w_i \in \{\prec,\prec'\}$ and $w_{i+1} \in \{\succ,\succ'\}$ for some $1\leq i < m+n$. Let $w^0$ be the word obtained from $w$ by interchanging elements at the $i$th  and $i+1$-st position. Then $sh(w^0)$ is obtained from $sh(w)$ by removing an outer corner and $|sh(w^0)|=k$.  By the inductive hypothesis we can sample this ``smaller'' process $\Lambda^0$ exactly (note that the parameter list $(X,Y)$ remains the same). We then perform an extra atomic step to add the missing corner to $sh(w^0)$ to obtain $sh(w)$. That is, if $\lambda, \kappa, \mu$ are the three partitions that sit at the inner corner of $sh(w^0)$ (like in Figure~\ref{atomic-steps}), we sample the outer corner $\nu$ based on the other three partitions using one of the four {\tt sample\{HH,HV,VH,VV\}} procedures (chosen based on $(w_i,w_{i+1})=(\prec,\succ),(\prec,\succ'),(\prec',\succ)$ or $(\prec',\succ'),
$ respectively). The atomic step exactly samples this corner so that it fits into the new process correctly (everything depends only on the three partitions involved and the type of atomic step). Thus we obtain $\Lambda$ sampled exactly and correctly from $\Lambda^0$ by replacing (in $\Lambda^0$) the partition $\kappa$ with the partition $\nu$.  
\end{proof}

\begin{rem} In the particular case of domino tilings of the Aztec
  diamond of size $n$, for which $w =(\prec', \succ)^n$, our algorithm
  coincides with the domino shuffling algorithm
  \cite[Section~6]{eklp2}, provided that we grow the staircase
  partition $sh(w)=(n,n-1,\ldots,1)$ diagonal by diagonal. More
  precisely, we shall sample the $\tau(i,j)$ by increasing order of
  $i+j$, so that after $k(k+1)/2$ steps ($k=0,\ldots,n$), we have
  sampled the Schur process corresponding to the word
  $w^{(k)}=(\succ)^{n-k} (\prec', \succ)^{k} (\prec')^{n-k}$ hence to
  the encoded shape $sh(w^{(k)})=(k,k-1,\ldots,1,0,\ldots,0)$. This
  Schur process corresponds to tilings of the Aztec diamond of size
  $k$. Using the dimer interpretation of the {\tt sampleHV} method
  shown in Figure~\ref{fig:shuffling}, we see that the procedure for
  passing from $w^{(k)}$ to $w^{(k+1)}$ coincides with the domino
  shuffling bijection between diamonds of sizes $k$ and $k+1$. Let us
  mention that another description of the shuffling algorithm in terms
  of nonintersecting paths, equivalent to our particle/hole
  description, is given in \cite{nor}.
\end{rem}

\begin{rem}
  As is clear from the construction, any path going right and down
  from the vertical axis to the horizontal axis that appears in the
  construction of the Schur process encoded by $sh(w)$ (i.e., a
  subpartition of $sh(w)$ in Figure~\ref{schur-proc}) will itself
  encode a Schur process. However, the process might not be related in
  a deeper way with the process encoded by $sh(w)$. For example, while
  $sh(w)$ may correspond to steep tilings, a subpartition of it may
  not. On the other hand, a subprocess on reverse plane partitions is
  also a reverse plane partition, or a staircase subprocess of the
  Aztec diamond is also an Aztec diamond.
\end{rem}

\begin{rem}
  One can modify the algorithm slightly for better space complexity
  (but the same time complexity). At any time it is is enough to store
  $m+n+1$ partitions where $m=sh(w)_1$ and $n=\ell(sh(w))$. Indeed,
  one can start with the $m+n+1$ red dots depicting empty partitions
  in Figure~\ref{GetType} and sample using the local rules, but update
  partitions in place: that is, for example $\tau(1,1)$ would be
  sampled and then overwritten in the place of $\tau(0,0)$ since the
  latter is not needed anymore. Similarly $\tau(1,2)$ would be sampled
  and then overwritten in the place of $\tau(0,1)$ using the same
  logic, and so on.
\end{rem}

The explicit description of the algorithm immediately allows us to estimate its time complexity, which is random and depends on the output.  We assume that we can sample Bernoulli and geometric random variables in time $O(1)$. 
We have:
\begin{prop}\label{prop:comp}
 The time complexity of {\tt SchurSample} is $O(|sh(w)| L)$, where we recall that $sh(w)$ is the encoding shape of the Schur process, and where $L:=\max \{\tau(l_i)_1, \ell(\tau(l_i)), i\in[0\dots m+n] \}$ is the maximum of the maximum part and the maximum number of parts of the partitions in the output sequence~$\Lambda$. 
\end{prop}
\begin{proof}
This is clear by induction on $|sh(w)|$. Indeed, we have seen that the complexity of adding a box to $sh(w)$ (Cauchy case and dual Cauchy case) is $O(\max(\ell(\nu), \nu_1))$ where $\nu$ is the top-right partition in the added box, which is bounded by $O(L)$. 
\end{proof}

\begin{rem}
  In some cases we can deterministically bound the quantity $L$ above,
  thus giving a deterministic, not output-sensitive, bound on the
  complexity. Suppose we sampled a Schur process of word $w$, and
  suppose that in doing so we never had to use both the {\tt sampleHH}
  and the {\tt sampleVV} procedure (this fact of course only depends
  on the word $w$).  Plane partitions and Aztec diamonds are examples
  of such Schur processes. For simplicity, suppose we never had to use
  the {\tt sampleVV} procedure (we may have used the {\tt sampleHH}
  procedure multiple times though). Then under these assumptions, the
  time complexity of {\tt SchurSample} is $O(|sh(w)| m)$ where $m$ is
  the total number of $\{\prec, \prec'\}$ elements in $w$.
\end{rem}

We now address entropy minimality/optimality. Consider a sample of the
Schur process with word $w$ and encoded shape $sh(w)$. We claim that
the random inputs needed to obtain it using {\tt SchurSample} can be
reconstructed. One simply has to proceed backwards, removing boxes
from the encoded shape one at a time, and noting that the {\tt
  sample\{HH,HV,VH,VV\}} methods can be reversed (since they are based
on bijections) to recover the value of the geometric or Bernoulli
random variable that they use. Thus we have:

\begin{prop}\label{prop:entropy}
 Assuming that we can sample sequences of independent geometric random variables and random bits optimally, the algorithm {\tt SchurSample} is entropy optimal.
\end{prop}

\begin{rem}
  In the case of $(m \times n)$-boxed plane partitions (encoded shape
  is $m^n$), the {\tt SchurSample} algorithm coincides with the RSK
  correspondence \cite{ges}. The latter takes an $m \times n$ array of
  nonnegative integers (corresponding to the values of the geometric
  random variables used in our algorithm) and maps it to a pair of
  semi-standard Young tableaux of common shape $\lambda$. One then
  takes these two tableaux and conjoins them (view each tableau as a
  sequence of interlacing partitions starting from $\emptyset$ and
  ending in $\lambda$) to obtain a plane partition (with central slice
  given by $\lambda$) which is the sampled plane partition.
\end{rem}

\begin{rem}
  The algorithm above can be thus viewed as a generalized RSK
  correspondence. It takes a word $w$ and a partial matrix
  $(S_{ij})_{(i,j) \in sh(w)}$ of (geometric/Bernoulli) integers and
  produces a sequence of partitions which interlace or dually
  interlace at every step. It reduces to RSK as remarked above, but
  also to dual RSK, see also \cite{ges, kra, pp}.
\end{rem}

Below we present some samples using the given algorithm. In Figure~\ref{large-pp-ad} we present a large plane partition (a) and Aztec diamond (b). The deterministic asymptotic shape is a law of large numbers (see \cite{or} for the case of the plane partition and \cite{cjp} for the Aztec diamond). In Figure~\ref{fig:pyramidCusp}, we show a large pyramid partition. Its width and the parameter $q$ are tuned so that its apparent limit shape exhibits two interesting cusp points. Finally, in Figure~\ref{ad-nonuniform} we present two large Aztec diamonds distributed according to a measure which is not the usual $q^{\text{Volume}}$. We can see (again) the appearance of cusps in the limit shape, as well as nodes (technically, two cusps coming ``very close'' to one another). This will be studied theoretically in future work.

\begin{figure}[h!]
\centering
\begin{subfigure}[b]{0.37\textwidth}
\includegraphics[scale=0.05]{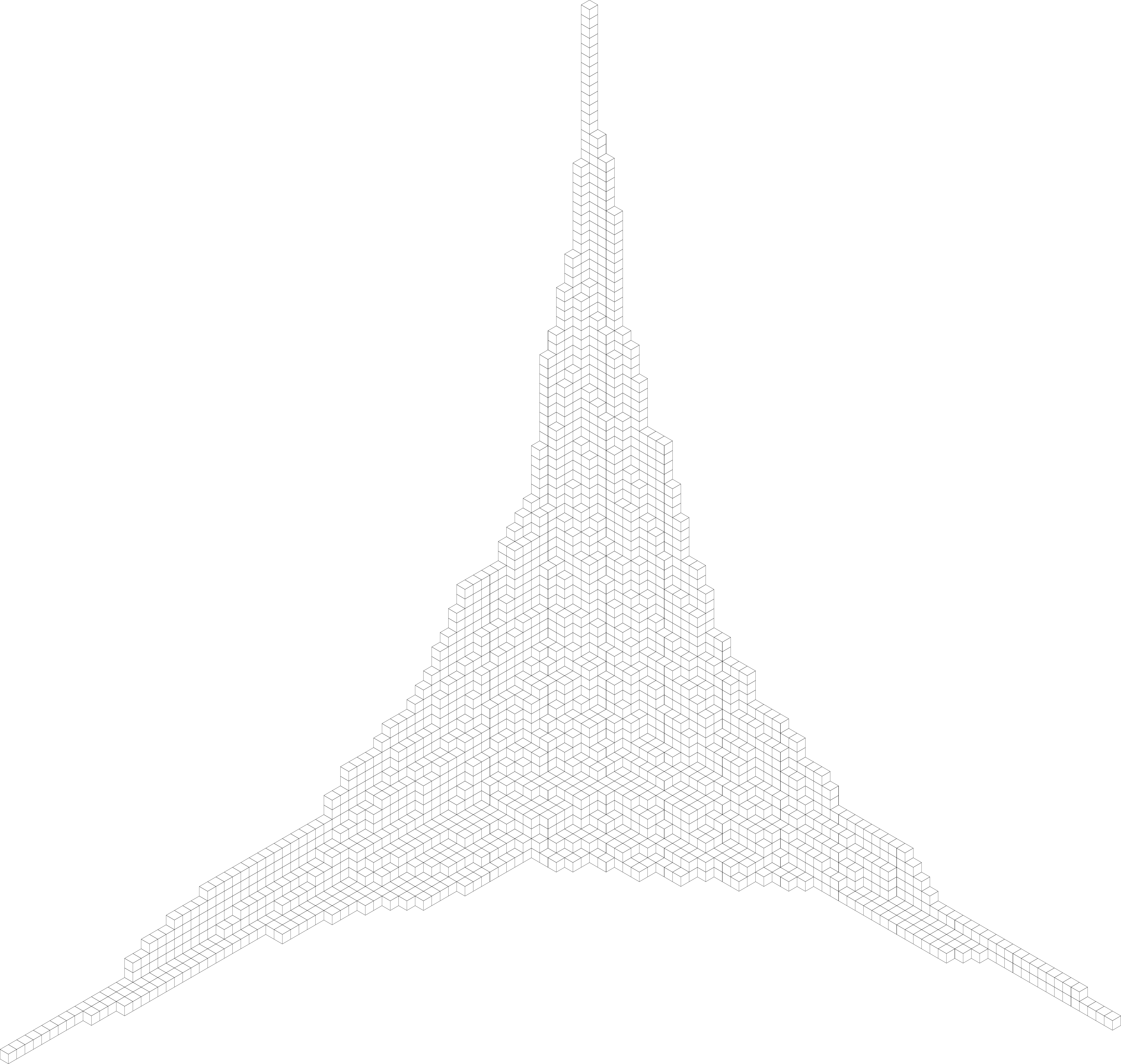}
\caption{}
\end{subfigure}
 \qquad
 \begin{subfigure}[b]{0.362\textwidth}
\includegraphics[scale=0.09]{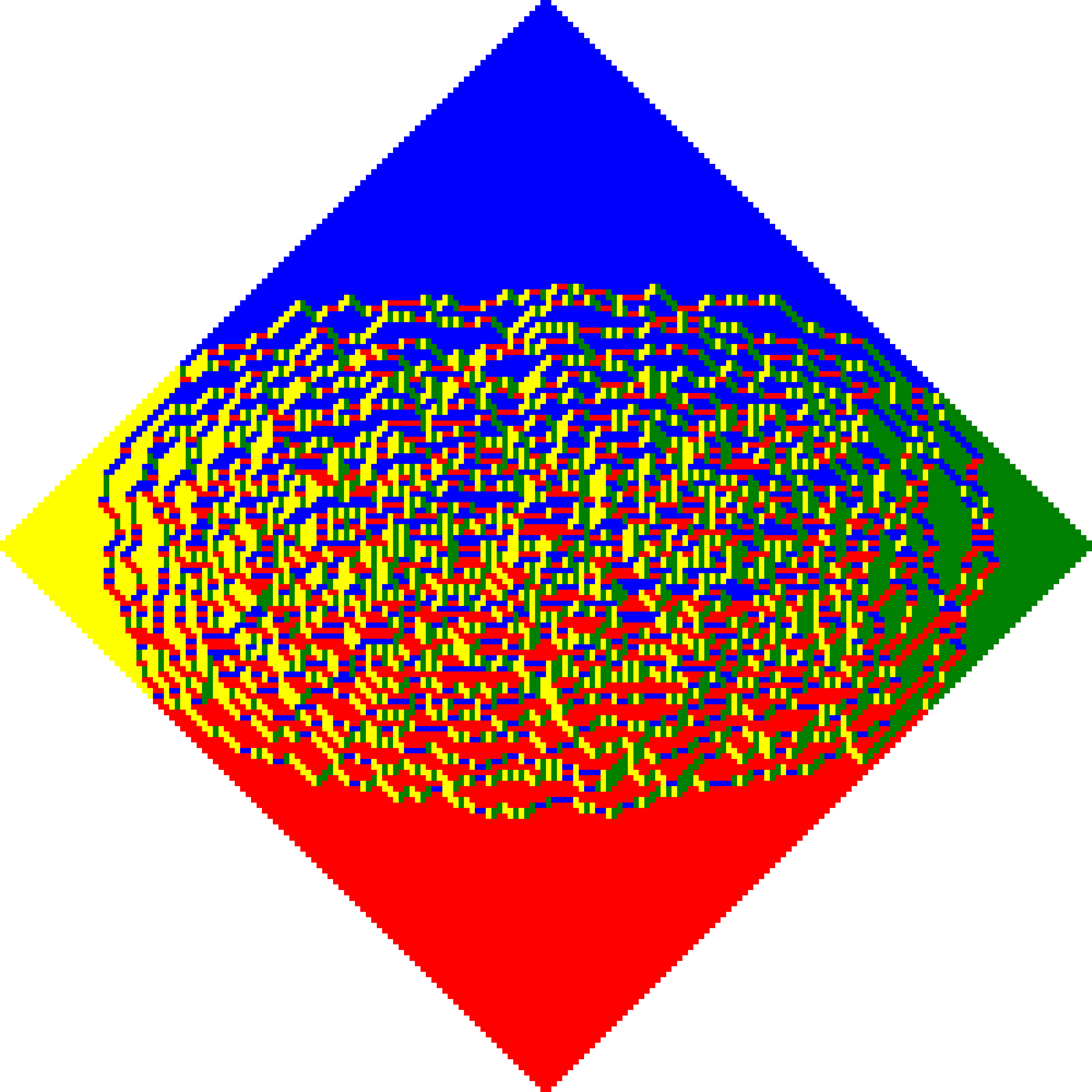}
\caption{}
\end{subfigure}
\caption{\fs{A random large plane partition with base contained in a $100\times100$ box and $q=0.93$ (a) and a random $100\times100$ Aztec diamond with $q=0.99$ (b). Both exhibit deterministic limit shapes. The limit shape is known in the second case as the ``arctic circle'' (the terminology comes from the uniform case $q=1$, in which case the frozen boundary is indeed a circle).}}
\label{large-pp-ad}
\end{figure}

\begin{figure}[!ht]
\includegraphics[scale=0.350]{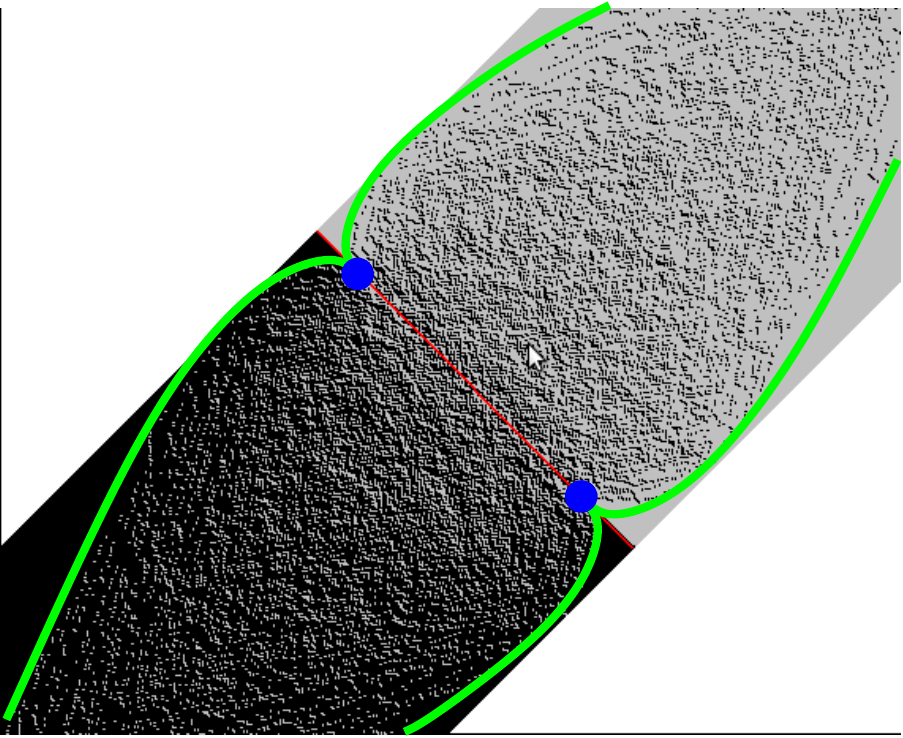}
\caption{\fs{A random pyramid partition of width $100$ with parameter $q=.99$ (only the particles of the corresponding Maya diagrams are displayed, not the dominoes). The apparent limit shape (in green) seems to exhibit cusp points (in blue), an interesting phenomenon that will be the subject of future work.}} 
\label{fig:pyramidCusp}
\end{figure}

\begin{figure}[!ht]
\centering
\begin{subfigure}[b]{0.425\textwidth}
\includegraphics[scale=0.033]{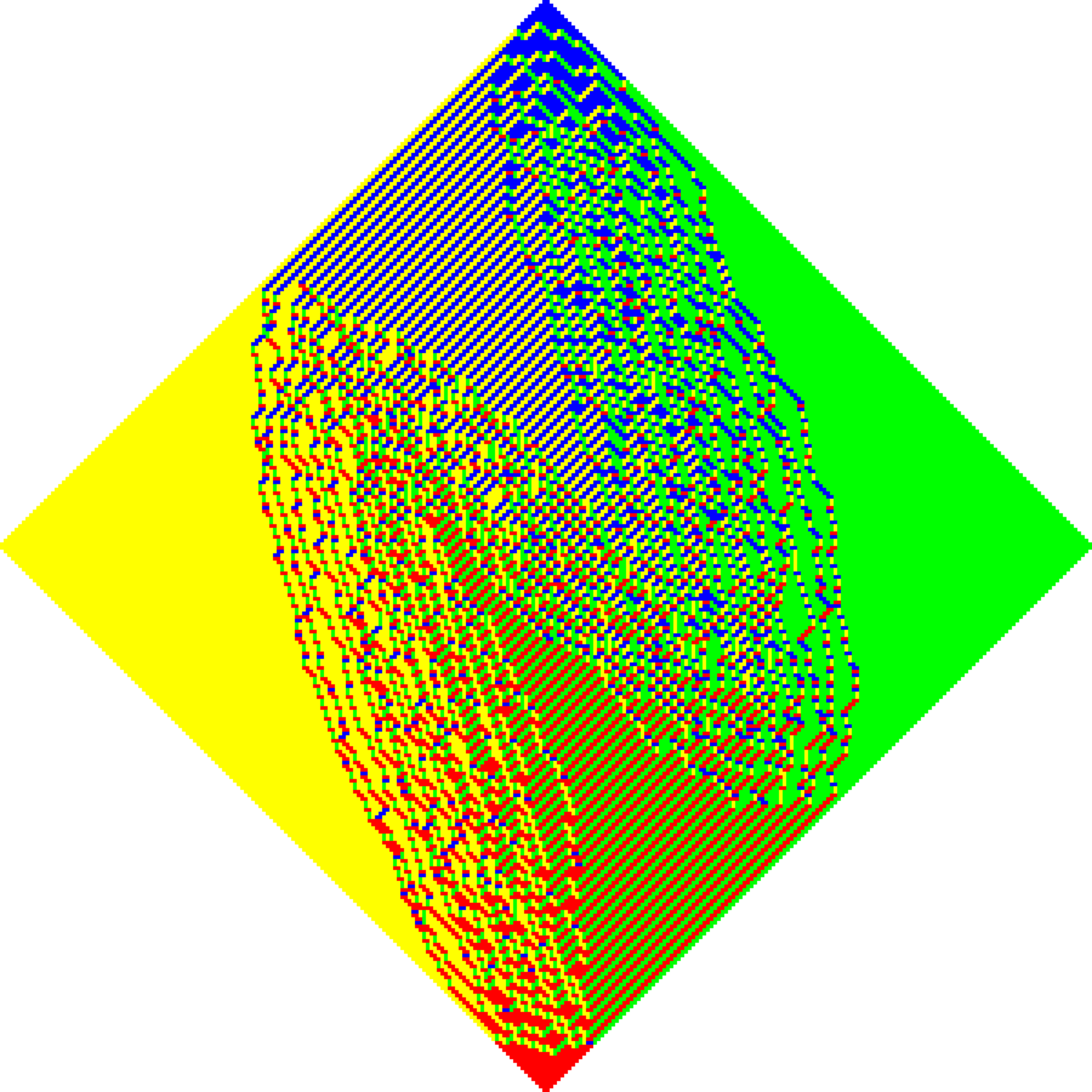}
\caption{}
\end{subfigure}
 \qquad
 \begin{subfigure}[b]{0.428\textwidth}
\includegraphics[scale=0.1]{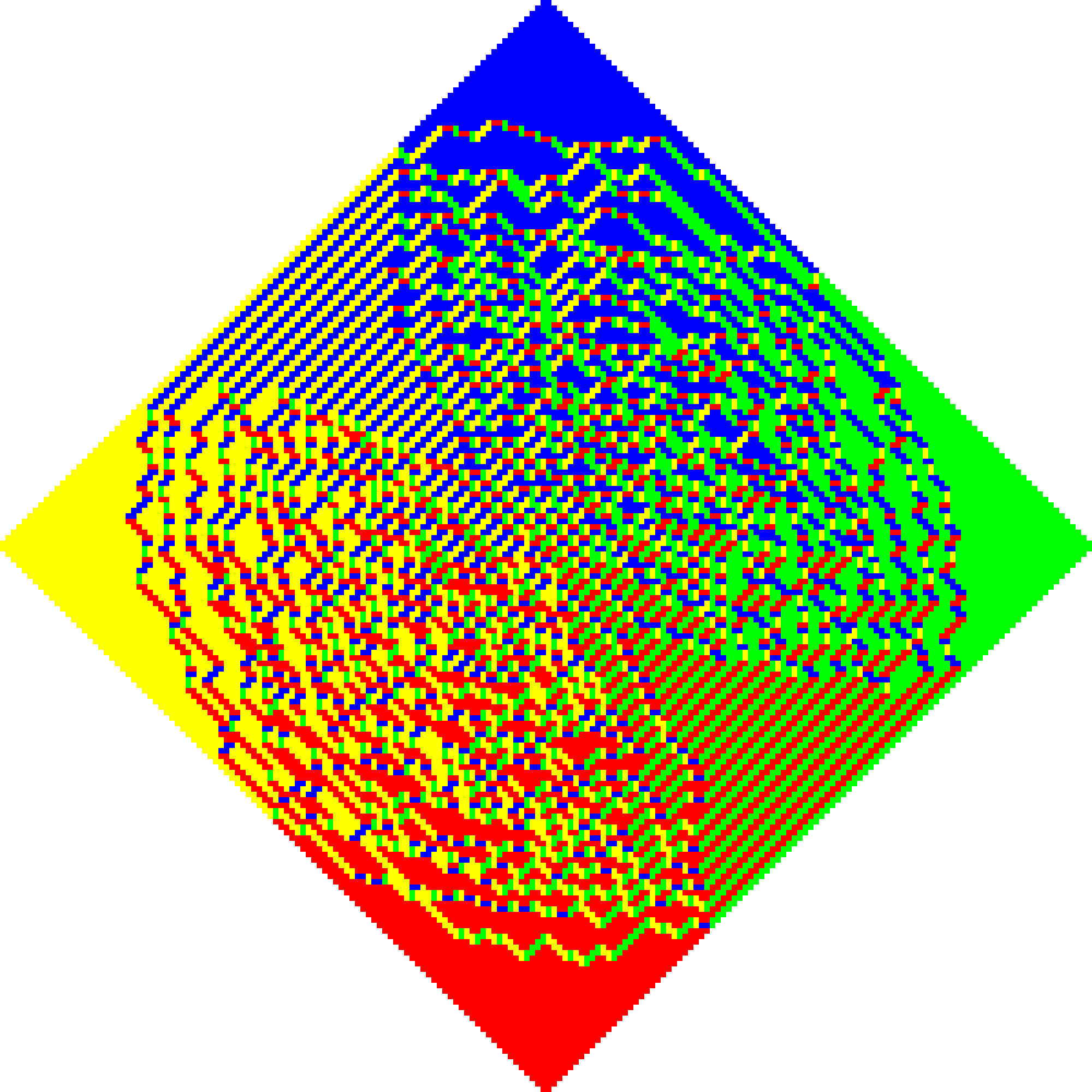}
\caption{}
\end{subfigure}
\caption{\fs{A random $150 \times 150$ Aztec diamond with parameters $x_{2i-1} = 48, x_{2i} = 1/2, y_{2i-1} = 16, y_{2i} = 1/8, \ \forall i \geq 1$, exhibiting the formation of cusps (a), and a random $100 \times 100$ Aztec diamond with parameters $x_{2i-1} = y_{2i-1} = 20, x_{2i} = y_{2i} = 1/20, \ \forall i \geq 1$ where the two cusps have coalesced (b).}}
\label{ad-nonuniform}
\end{figure}

\section{Symmetric Schur processes} \label{sec:symm}

\subsection{Basic definition and examples}

We now turn our attention to Schur processes on interlaced sequences with a free boundary, i.e.\ where the end partition is not fixed.  Precisely, for a word $w\in\{\prec,\succ,\prec',\succ'\}^n$ we will define the process on \emph{right-free $w$-interlaced} sequence of partitions, i.e.\ $\Lambda=(\emptyset=\lambda(0),\lambda(1),\dots, \lambda(n))$ such that $\lambda(i-1)w_i\lambda(i)$, for $i=1,\dots,n$. This is a one-sided free boundary  process, or mixed boundary process,  since one side of the sequence is fixed (the empty partition), while the other one is free.  One can, of course, define a double-sided free boundary Schur process, but we only provide a sampling algorithm for the one-sided free boundary process. 

\begin{definition}
For a word $w=(w_1,w_2,\dots,w_n) \in \{\prec,\succ,\prec',\succ'\}^n$, the \emph{right-free Schur process} of word $w$ with parameters $(Z;t) = (z_1, \dots, z_n;t)$ is the measure on the set of right-free $w$-interlaced sequences of partitions $\Lambda = (\emptyset=\lambda(0), \lambda(1), \dots, \lambda(n) = \lambda)$ given by 
\begin{align} \label{schur-measure-symm}
 Prob(\Lambda) \propto t^{|\lambda_n|}\prod_{i=1}^n z_i^{||\lambda(i)|-|\lambda(i-1)||}.
\end{align}
\end{definition}

\begin{rem}
  Note that we could have omitted the term with parameter $t$ and keep
  the full generality as the Schur process with parameter $(Z;t)$ is
  the same as the one with parameter $(\overline{Z};1)$ where
  $\overline{z}_i=t^{\epsilon_i}z_i$ and $\epsilon_i=1$ if
  $w_i\in\{\prec,\prec'\}$ and $-1$ otherwise, for $i=1,\dots,n$.
\end{rem}

We will also here refer to the right-free Schur process as a symmetric Schur process since it can be defined as a measure on symmetric sequences of length $2n+1$ starting and ending with the empty partition, i.e.\  
\begin{equation}
  (\emptyset=\lambda(0), \lambda(1), \dots, \lambda(n-1), \lambda(n) = \lambda, \lambda(n-1), \dots, \lambda(1), \lambda(0)=\emptyset )\label{eq:symmSP}
\end{equation}
with probability proportional to
\begin{equation} \label{schur-measure-symm2}
 \prod_{i=1}^{2n+1} t_i^{||\lambda(i)|-|\lambda(i-1)||}
\end{equation}
where $t_it_{2n-i+1}=t^{\epsilon_i}z_i$, for $i=1,\dots,n$.

We now give examples of right-free (or symmetric) interlaced sequences on which the Schur process described above can be defined. These examples correspond to symmetric tilings of a plane.  

\paragraph{Symmetric reverse plane partitions}  of symmetric shape $\mathcal{S}$, are reverse plane partitions of shape $\mathcal{S}$ with symmetric fillings, i.e.\ $\pi_{i,j}=\pi_{j,i}$. Of course we assume the shape $\mathcal{S}$ to be symmetric, namely that $(i,j) \in \mathcal{S}$ implies $(j,i)\in \mathcal{S}$. They are in bijection with symmetric $w$-interlaced sequences where $sh(w)=\mathcal{S}$. Special case when $\mathcal{S}$ is the Young diagram of $n^n$ corresponds to symmetric $(n \times n)$-boxed plane partitions. Symmetric reverse plane partitions correspond to symmetric planar lozenge tilings.

\paragraph{Plane overpartitions} \cite{ccv} and, more generally, \emph{one-sided free boundary steep tilings} \cite{bcc} are examples of right-free (or equivalently symmetric) domino tilings  of a plane. A plane overpartition is a plane partition where in each row the last occurrence of an integer can be overlined or not and all the other occurrences of this integer are not overlined, also in each column the first occurrence of an integer can be overlined or not and all the other occurrences of this integer are overlined. An example of a plane overpartition is given in Figure \ref{plane-overpartitions}.
\begin{figure}[ht]
    \includegraphics{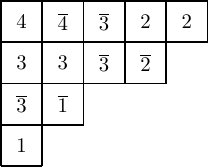}
    \caption{\fs{A plane overpartition partition corresponding to the interlacing sequence $\emptyset \prec (1) \prec' (2) \prec (2,2) \prec' (3,3,1) \prec (5,3,1) \prec' (5,4,1) \prec (5,4,1,1) \prec' (5,4,2,1)$.}}  
    \label{plane-overpartitions}
\end{figure}
 
Plane overpartitions whose largest filling is at most $n$ are in correspondence with right-free $w$-interlaced sequences where $w=(\prec,\prec')^n$. Indeed, starting from a plane partition of shape $\lambda$ with the largest part at most $n$ we can form a sequence of interlaced partitions 
\[
\emptyset=\lambda(0) \prec \lambda(1) \prec'\lambda(2)\prec\cdots \prec \lambda(2n-1) \prec' \lambda(2n)=\lambda
\]
where $\lambda(i)$ is the partition whose shape is formed by all fillings greater than ${n-i/2}$, where the convention is that $\overline{k}=k-1/2$. From this it is clear that they can also be viewed as symmetric pyramid partitions.

Plane overpartitions are a special case of a one-sided free boundary steep tilings, i.e.\ domino tilings which are in correspondence with right-free $w$-interlaced sequences where $w_{2i}\in \{\prec,\succ\}$ and $w_{2i+1}\in \{\prec',\succ'\}$. See \cite{bcc} for details.

\subsection{Sampling algorithm}

We now give the exact sampling algorithm for the symmetric Schur process of word $w$ with parameters $(Z;1)$. The algorithm is a modification of {\tt SchurSample}.

Let $w^{sym}=w \cdot w^*$, where $w^*$ is obtained by reversing the word $w$ first and then by inverting every element in the word, where by inverting we mean replacing $\succ$ (respectively $\succ'$) with $\prec$ (respectively $\prec'$), and vice versa. Also, let $Z^{sym}=Z\cdot Z^r$ where $Z^r$ is the reverse of $Z$.    For example, for the word $w=(\prec,\succ')$ and $Z=(z_1,z_2)$, $w^{sym}=(\prec,\succ',\prec',\succ)$  and $Z^{sym}=(z_1,z_2,z_2,z_1)$. Set $(\pi,X,Y, \text{getType})={{\tt Par}}(w^{sym},Z^{sym})$  where $\tt{Par}$ is defined in the pre-computation step of  $\tt{SchurSample}$. Note that in this case the number of $x$'s and $y$'s is the same and equal to the length of the word $w$, say $n$, and that $x_i=y_i$, for $i=1,\dots,n$.

The algorithm that produces samples form the symmetric Schur process is obtained from {\tt SchurSample} for word $w^{sym}$ and parameters $Z^{sym}$ with the following two alternations. Since the output needs to be symmetric, at each step we need to make sure that the partitions corresponding to off-diagonal boxes $(i,j)$ and $(j,i)$ are equal. This just means that we need to add two boxes to the symmetric shape at the same time and use the same sample for both boxes to construct the corresponding partition. The sample will come from either $Geom(x_iy_j)$ or $Bernoulli(x_iy_j)$ depending on the type of the box $(i,j)$. Another difference is that  the partitions corresponding to the diagonal boxes $(i,i)$ will be constructed using a sample from $Geom(x_i)=Geom(\sqrt{x_iy_i})$. The algorithm is given below.
\newline

\indent {\fs \textbf{Algorithm} {\tt SymmetricSchurSample}} \\
\indent {\fs \textbf{Input:} $\pi$, partitions $\tau(0,i) = \emptyset = \tau(j,0)$ and parameters $x_i = y_i$, $0 \le i,j \le n$\\
\indent \textbf{for} $j=1 \dots \ell(\pi)$ \\
\indent\indent \textbf{for} $i=1 \dots \pi_j$ \\
\indent\indent\indent \textbf{if $j > i$} \\
\indent\indent\indent\indent $type$ = getType($i,j$) \\
\indent\indent\indent\indent $\tau(i,j)$=sample($\tau(i-1,j),\tau(i,j-1),\tau(i-1,j-1), x_i y_j, type$) \\
\indent\indent\indent \textbf{elseif $j == i$} \\
\indent\indent\indent\indent $type$ = getType($i,i$) \\
\indent\indent\indent\indent $\tau(i,i)$=sample($\tau(i-1,i),\tau(i-1,i),\tau(i-1,i-1), x_i, type$) \\
\indent\indent\indent \textbf{elseif $j < i$} \\
\indent\indent\indent\indent $\tau(i,j) = \tau(j,i)$ \\
\indent \textbf{Output:} The sequence of partitions $\Lambda$ defined by \[\Lambda := (\emptyset = \tau(l_0), \tau(l_1),\dots,\tau(l_n) = \lambda,\dots,\tau(l_{2n-1}) = \tau(l_1),\tau(l_{2n})=\tau(l_0)=\emptyset),\]
\noindent where $(l_0 = (0,n), l_1 ,\dots,l_{2n-1}, l_{2n} = (n,0))$ is the \emph{ordered} sequence of lattice points on the boundary of $\pi$ clockwise from the vertical to the horizontal axis.
}

Note that the diagonal elements in the algorithm above are produced either using {\tt sampleHH} or {\tt sampleVV} procedure where $\lambda=\mu$ in the notation given in the description of the bijections. So, they are produced using the following bijection  $(\kappa,G) \mapsto \nu$  where 
\begin{equation} \label{bij-H} 
  \nu_i =
  \begin{cases}
    \mu_1 + G & \text{if $i=1$} ,\\
    \mu_i + \mu_{i-1} -
    \kappa_{i-1} & \text{if $i>1$}.\\
  \end{cases}
\end{equation}
This is a bijective mapping from $\mu$ and $G \in
\mathbb{N}$ such that $\mu \succ \kappa$ to $\nu$ such that $\mu \prec \nu$, for a fixed  $\mu$. In addition, we have that $2|\mu| + G = |\kappa|+|\nu|$.  
Thus, the diagonal elements of type HH are produced using
\newline
\indent {\fs \textbf{def} {\tt sampleH}($\mu, \kappa,x$) \\
\indent \indent sample $G \sim Geom(x)$ \\
\indent \indent construct $\nu$  as in (\ref{bij-H})\\
\indent \indent \textbf{return} $\nu$}\\
The diagonal elements of type VV are obtained by conjugating  {\tt sampleH}($\mu', \kappa',x$).
The bijection described above allow us to prove the Littlewood identity \cite[p.~93]{mac}
\begin{equation*}
  \sum_{\nu} s_{\nu / \mu} (X) = \prod_{i} \frac{1}{1-x_i} \prod_{i<j} \frac{1}{1-x_i x_j} \sum_{\kappa} s_{\mu / \kappa} (X), \label{eq:littlewood} 
\end{equation*}
or, equivalently, the reflection relations \cite{bcc}
\begin{align}
 \Gapl(z) |\underline{t}\rangle  &= \frac{1}{1-zt} \Gami(z t^2) |\underline{t}\rangle, \\
 \Gatpl(z) |\underline{t}\rangle  &= \frac{1}{1-zt} \Gatmi(z t^2) |\underline{t}\rangle, 
\end{align}
where 
\begin{equation}
  |\underline{t}\rangle  = \sum_{\nu} t^{|\nu|} | \nu \rangle.\label{eq:tdef}
\end{equation}
The proof follows from:
\begin{equation}
  \sum_{\substack{\nu \\  \mu \prec \nu}} z^{|\nu|-|\mu|} t^{|\nu|} = \sum_{\substack {\kappa \\ \mu \succ \kappa}} \sum_{G \geq 0} z^{|\mu|-|\kappa| + G} t^{2 |\mu| -|\kappa|+G} = \frac{1}{1-zt}\sum_{\substack{\kappa \\ \mu \succ \kappa}} (zt^2)^{|\mu|-|\kappa|} t^{|\kappa|}.\label{eq:littleproof}
\end{equation}
The partition function of the symmetric Schur process can then be written as:
\begin{equation}
  Z^{sym}_w = \sum_{\Lambda}t^{|\lambda(n)|}\prod_{i=1}^n z_i^{||\lambda(i)|-|\lambda(i-1)||}= \langle \emptyset | \prod_{i=1}^n\Gamma_i(z_i)  |\underline{t}\rangle,
  \label{eq:symmZ}
\end{equation}
where $\Gamma_i$ is defined as before, i.e.~it is $\Gapl,\Gami,\Gatpl,\Gatmi$ if $w_i$ is $\prec,\succ,\prec',\succ'$, respectively.

Using the commutation relation, reflection relations and action of $\Gamma$'s on the vacuum vectors we get the following proposition.
\begin{prop}
  \label{prop:Z-symm}
  The partition function of the right-free Schur process of word $w$
  with parameters $(Z;t)$ is
  \begin{equation}
    Z_w^{symm} = \prod_{i, w_i \in \{\prec, \prec'\}} \frac{1}{1-tz_i} \prod_{i<j, w_i \in \{\prec, \prec'\}, w_j \in \{\succ, \succ'\} } (1+ \epsilon_{i,j} z_i z_j)^{\epsilon_{i,j}} \prod_{i<j, w_i, w_j \in \{\prec, \prec'\}} (1+ \delta_{i,j} t^2z_i z_j)^{\delta_{i,j}}
  \end{equation}
  where $\epsilon_{i,j} = 1$ if
  $(w_i, w_j) \in \{ (\prec, \succ'), (\prec', \succ)\}$ and
  $\epsilon_{i,j} = -1$ otherwise, $\delta_{i,j} = -1$ if $w_i = w_j$
  and $\delta_{i,j}=1$ otherwise.
\end{prop}

\subsection{Even row and even column symmetric Schur processes}

One can introduce versions of the symmetric Schur process by requesting that the free partition has all even parts or that its conjugate has even parts. We call these the \emph{even rows} and \emph{even columns symmetric Schur processes}.

One can derive the partition functions for such processes using the Cauchy identities and the modified versions of the Littlewood identity \cite[p.~93]{mac}
\begin{align}
  \sum_{\nu \ \text{with even  rows}} s_{\nu / \mu} (X) &= \prod_{i\leq j} \frac{1}{1-x_i x_j} \sum_{\kappa \ \text{with even rows}} s_{\mu / \kappa} (X), 
  \label{littlewood-er}\\
  \sum_{\nu \ \text{with even  columns}} s_{\nu / \mu} (X) &= \prod_{i<j} \frac{1}{1-x_i x_j} \sum_{\kappa \ \text{with even  columns}} s_{\mu / \kappa} (X). \label{littlewood-ec}
\end{align}
They are equivalent to the reflection relations \cite{bcc}
\begin{align}
   \Gapl(z) |\underline{t}^{er}\rangle &= \frac{1}{1-(zt)^2} \Gami(zt^2) |\underline{t}^{er}\rangle ,\label{Littlewooder1}\\
   \Gatpl(z) |\underline{t}^{er}\rangle & = \Gatmi(z t^2) |\underline{t}^{er}\rangle  \label{Littlewooder2},\\
 \Gapl(z) |\underline{t}^{ec}\rangle&= \Gami(zt^2)v^{ec} (t),\label{Littlewoodec1}\\
 \Gatpl(z) |\underline{t}^{ec}\rangle &=  \frac{1}{1-(zt)^2} \Gatmi(z t^2) |\underline{t}^{ec}\rangle ,\label{Littlewoodec2}
\end{align}
where 
\begin{equation}
  |\underline{t}^{er}\rangle=\sum_{\nu\ \text{with  even  rows}}t^{|\nu|} |\nu \rangle \quad \text{and} \quad |\underline{t}^{ec}\rangle=\sum_{\nu\ \text{with  even columns}}t^{|\nu|} |\nu \rangle.
  \label{eq:erecstate}
\end{equation}

\begin{prop}
  \label{prop:Z-symmeven}
  The partition functions for the even rows/columns symmetric
  Schur processes read
  \begin{align}
    Z_w^{symm, \ er} & = \prod_{i, w_i=\prec} \frac{1}{1-(tz_i)^2} \prod_{i<j, w_i \in \{\prec, \prec'\}, w_j \in \{\succ, \succ'\} } (1+ \epsilon_{i,j} z_i z_j)^{\epsilon_{i,j}} \prod_{i<j, w_i, w_j \in \{\prec, \prec'\}} (1+ \delta_{i,j} t^2z_i z_j)^{\delta_{i,j}}, \\
    Z_w^{symm, \ ec} & = \prod_{i, w_i=\prec'} \frac{1}{1-(tz_i)^2}\prod_{i<j, w_i \in \{\prec, \prec'\}, w_j \in \{\succ, \succ'\} } (1+ \epsilon_{i,j} z_i z_j)^{\epsilon_{i,j}} \prod_{i<j, w_i, w_j \in \{\prec, \prec'\}} (1+ \delta_{i,j} t^2z_i z_j)^{\delta_{i,j}}
  \end{align}
  where $\epsilon_{i,j} = 1$ if
  $(w_i, w_j) \in \{ (\prec, \succ'), (\prec', \succ)\}$ and
  $\epsilon_{i,j} = -1$ otherwise, $\delta_{i,j} = -1$ if $w_i = w_j$
  and $\delta_{i,j}=1$ otherwise.
\end{prop}

{\tt SymmetricSchurEvenRowsSample} and {\tt SymmetricSchurEvenColumnsSample} algorithms are obtained from {\tt SymmetricSchurSample} by changing the sampling of  diagonal elements, i.e.\ by changing {\tt sampleH} and {\tt sampleV}.
We start with the even row case.  We need to make minor changes to make sure we produce even partitions, i.e.\ partitions with even rows. In this case we need a bijection that maps  pairs of even partitions $\kappa$, where $\mu \succ \kappa$, and $G \in \mathbb{N}$ to even partitions $\nu$ such that $\mu\prec \nu$. We can achieve this by setting
 \begin{equation} \label{bij-HER}
  \nu_i =
  \begin{cases}
    2\lceil\mu_1/2\rceil + 2G & \text{if $i=1$}, \\
    2\lceil\mu_i/2\rceil + 2\lfloor\mu_{i-1}/2\rfloor -
    \kappa_{i-1} & \text{if $i>1$}.\\
  \end{cases}
\end{equation}
The bijection satisfies $2|\mu| + 2G = |\kappa|+|\nu|$ and can be used to prove the Littlewood identity (\ref{littlewood-er}) or equivalently (\ref{Littlewooder1}) and (\ref{Littlewooder2}).  
{\tt SymmetricSchurEvenRowsSample} is obtained from {\tt SymmetricSchurSample} by replacing  {\tt sampleH} with {\tt sampleH$^{er}$}: \\
\indent {\fs \textbf{def} sampleH$^{er}$($\mu, \kappa,x$) \\
\indent \indent sample $G \sim Geom(x^2)$ \\
\indent \indent construct $\nu$  as in (\ref{bij-HER})\\
\indent \indent \textbf{return} $\nu$}

For the even column case, we need a bijection that maps even column partitions $\kappa$, where $\mu \succ \kappa$ to even column partitions $\nu$ such that $\mu\prec \nu$. We can achieve this by setting
\begin{equation} \label{bij-HEC}
  \nu_i =
  \begin{cases}
    \mu_1  & \text{if $i=1$} ,\\
    \mu_i + \mu_{i-1} -
    \kappa_{i-1} & \text{if $i>1$}.\\
  \end{cases}
\end{equation}
The bijection satisfies $2|\mu|  = |\kappa|+|\nu|$ and can be used to prove the Littlewood identity (\ref{littlewood-ec}) or equivalently (\ref{Littlewoodec1}) and (\ref{Littlewoodec2}).  
In this case we replace {\tt sampleH} with {\tt sampleH$^{ec}$}: \\
\indent {\fs \textbf{def} {\tt sampleH$^{ec}$}($\mu, \kappa,x$) \\
\indent \indent construct $\nu$  as in (\ref{bij-HEC})\\
\indent \indent \textbf{return} $\nu$}

Note that {\tt sampleH$^{ec}$} is deterministic since in this case we do not need to produce a random number $G$, as in previous cases. 

\begin{rem}
  Correctness, complexity analysis and the entropy
  optimality of the regular {\tt SchurSample} algorithm transfer,
  mutatis mutandis, to the symmetric cases.
\end{rem}

\begin{rem}
  Our sampling algorithm is based on an RSK-type correspondence between symmetric $w^{sym}$-interlaced sequences and symmetric fillings of $sh(w^{sym})$. Depending on the type of the box, the fillings are either nonnegative integers or elements of $\{0,1\}$.  Each box  $(i,j)$, where $i\leq j$ has a corresponding term in the partition function. The three groups of terms, see Proposition \ref{prop:Z-symm}, correspond to three groups of boxes (diagonal boxes, the upper part of $\pi$ corresponding to the  encoded shape of $w$, and the group of all other boxes i.e.,\ below the second group and to the left of the diagonal boxes). Even row or even column cases are similar except we use different bijections for diagonal boxes and in the even column case all diagonal boxes are filled with zeros.  The three algorithms for sampling symmetric Schur processes described above can be viewed as generalizations of the symmetric RSK \cite{knu,sta}. 
\end{rem}

In Figure~\ref{fig:sym_pp} we present, side by side, the output of sampling a plane partition with parameter $q = 0.953$ and a symmetric plane partition with parameter $q_s = q^2$. The two outputs are clearly comparable and one observes a similarity in the limit shapes, but note that the two $q$ parameters are related to each other by squaring (this is accounted for by the fact that we are sampling ``half'' as many geometric random variables in the symmetric case compared to the non-symmetric case). In Figure~\ref{spyrp} we present the particle configuration associated to an unbounded symmetric pyramid partition with parameter $q = 0.95$ and the plot (as a plane partition) of the associated plane overpartition tableau (see example in Figure~\ref{plane-overpartitions}), after the mapping $\overline{i} \mapsto 2i-1, i \mapsto 2i$. Note the plane partition is a strict plane partition in the terminology of \cite{vul2}. 

\begin{figure}[!ht]
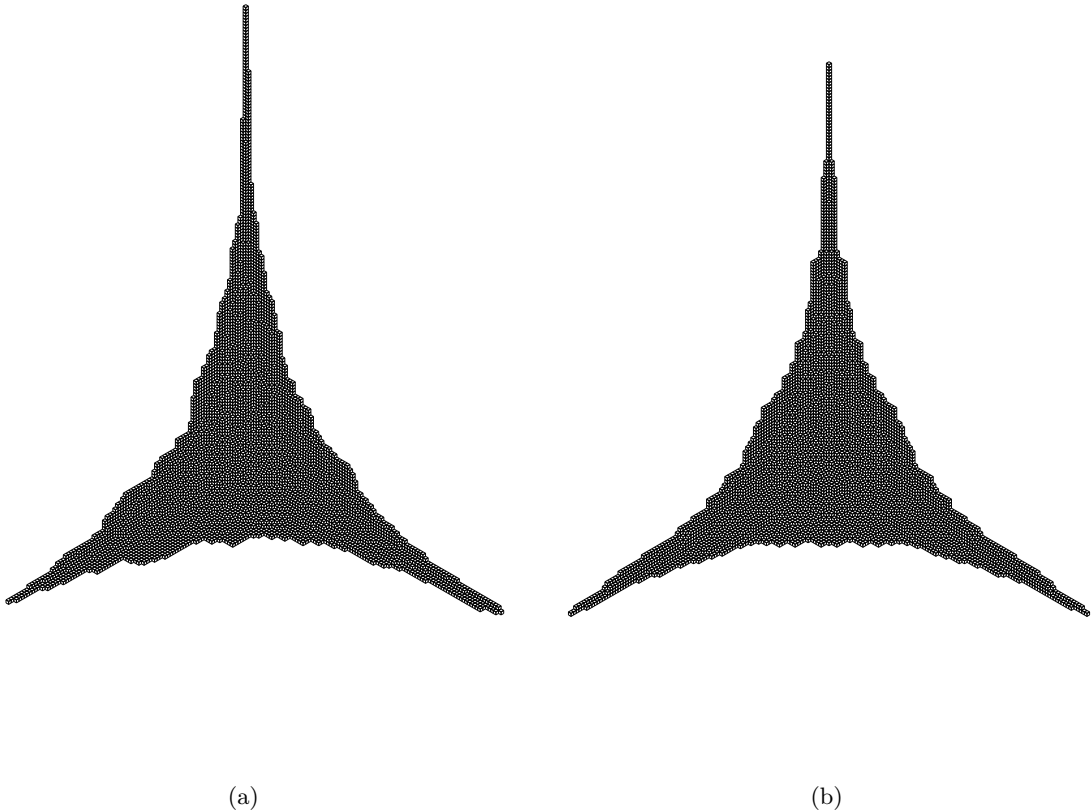

\begin{subfigure}[b]{0.42\textwidth}
\includegraphics[scale=0.04]{pp.pdf}  
\caption{}
\end{subfigure}
\qquad
\begin{subfigure}[b]{0.42\textwidth}
\includegraphics[scale=0.04]{pp_sym.pdf}
\caption{}
\end{subfigure}

\caption{\fs{A random plane partition for $q = 0.953$ (a) and a random symmetric plane partition for $q^2$ (b).}} 
\label{fig:sym_pp}
\end{figure}

\begin{figure}[!ht]
\begin{subfigure}[b]{0.43\textwidth}
\includegraphics[scale=0.5]{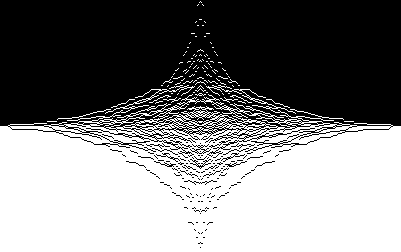}
\caption{}
\end{subfigure}
 \quad
 \begin{subfigure}[b]{0.45\textwidth}
\includegraphics[scale=0.05]{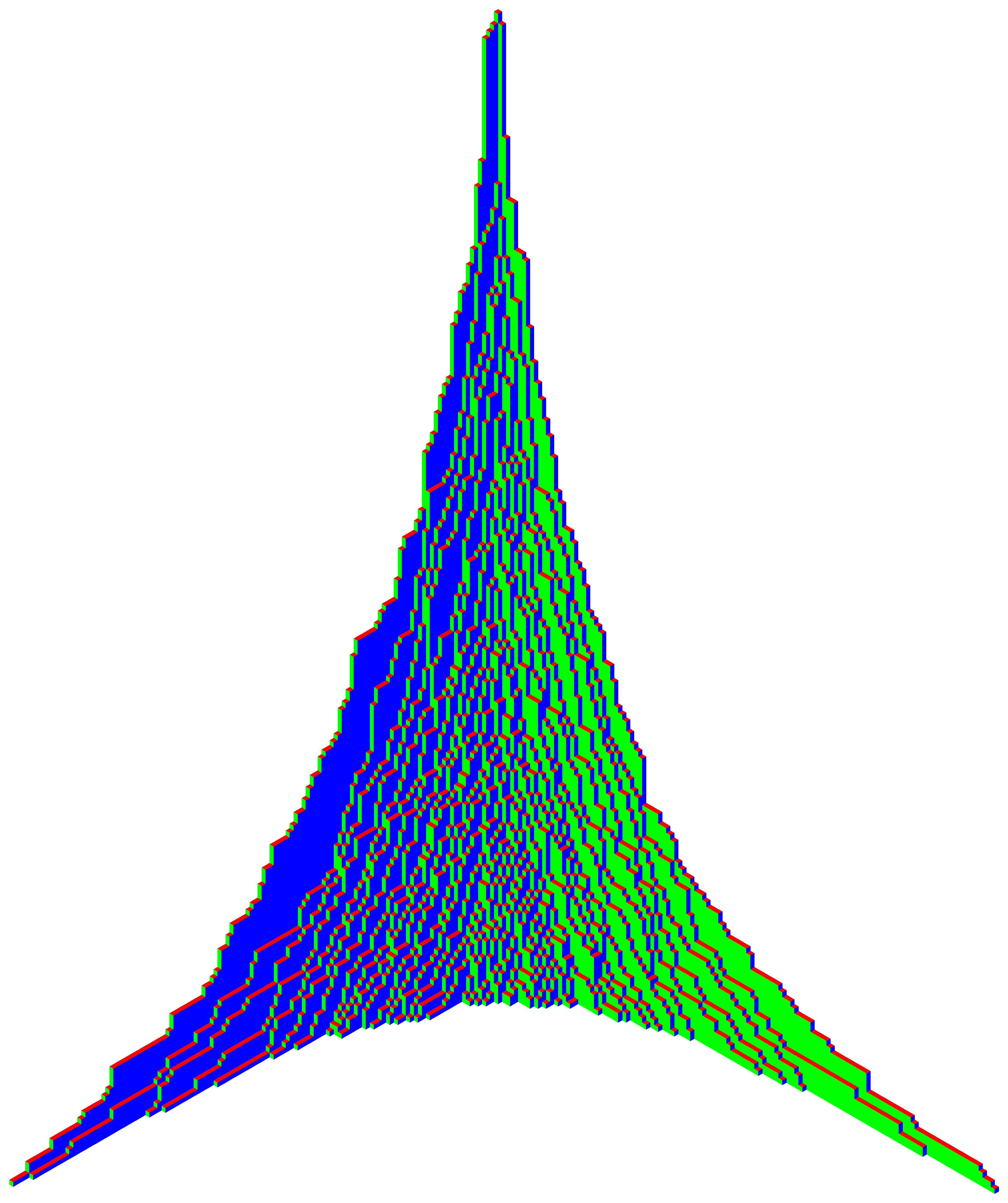}
\caption{}
\end{subfigure}
 \caption{\fs{An unbounded symmetric pyramid partition (only particles and holes depicted) (a) and the height function (as a plane partition) of the plane overpartition coming from (half) the associated symmetric Schur process (b), after we have replaced $\overline{i} \mapsto 2i-1, i \mapsto 2i$ in the tableau.}}
 \label{spyrp}
\end{figure}

\section{Schur processes with infinitely many parameters} \label{sec:unbounded}

In Definition~\ref{def:schurdef} we consider Schur processes depending
on a finite number of parameters: a finite word $w$ defining the
interlacing relations, and a sequence of weights $Z$ controlling the
size variation of the partitions. This is not the most general
definition of a Schur process as, for instance, one could consider
infinite sequences of interlaced partitions such as those coding for
(unboxed) plane partitions \cite{or}. In this section we explain how
to handle such a situation, and the most general Schur process will be
discussed in the next section.

For simplicity we will concentrate on the case of a bi-infinite
sequence of partitions $\Lambda=(\lambda(i))_{i \in \Z}$ forming a
``pyramid'' in the sense that
\begin{equation}
  \label{eq:infseq}
  \begin{cases}
    \lambda(i) \succ \lambda(i+1) \text{ or } \lambda(i) \succ'
    \lambda(i+1) & \text{for $i \geq 0$,} \\
    \lambda(i) \succ \lambda(i-1) \text{ or } \lambda(i) \succ'
    \lambda(i-1) & \text{for $i \leq 0$,} \\
    \lambda(i) = \emptyset & \text{for $|i|$ large enough.}
  \end{cases}
\end{equation}
In the terminology of Section~\ref{sec:schur}, we are considering a
$w$-interlaced sequence of partitions where $w$ is bi-infinite word
where no $\prec$ or $\prec'$ appears on the right of a $\succ$ or
$\succ'$. This situation includes unboxed plane partitions (when we
have only unprimed symbols) and pyramid partitions of arbitrary width
(when primed and unprimed symbols alternate). To each sequence
satisfying \eqref{eq:infseq}, we associate a weight
\begin{equation}
  \label{eq:infw}
  W(\Lambda) = \prod_{i \geq 0} a_i^{|\lambda(-i)|-|\lambda(-i-1)|}
  \prod_{j \geq 0} b_j^{|\lambda(j)|-|\lambda(j+1)|}
\end{equation}
where the $a$'s and $b$'s are nonnegative real parameters.  The most
natural specialization is again the $q^{\text{Volume}}$ measure which
is here obtained for instance by taking $a_i=b_i=q^{i+1/2}$ for all $i
\geq 0$. By adapting Proposition~\ref{prop:SPZexpr}, it is
straightforward to check that the partition function reads
\begin{equation}
  \label{eq:infZ}
  Z_w = \prod_{i,j \geq 0} (1 + \epsilon_{i,j} a_i b_j)^{\epsilon_{i,j}},
\end{equation}
where $\epsilon_{i,j} = 1$ if
$(w_{-i}, w_{j+1}) \in \{ (\prec, \succ'), (\prec', \succ)\}$ and
$\epsilon_{i,j} = -1$ otherwise. It is finite whenever
\begin{equation}
  \label{eq:condZ}
  \sum_{i \geq 0} ( a_i + b_i ) < \infty,
\end{equation}
so that normalizing \eqref{eq:infw} yields a probability measure over
sequences satisfying \eqref{eq:infseq}, which we call the
\emph{pyramidal Schur process} of parameters $(a_i,b_i)_{i \geq
  0}$. Note that the marginal law of $\lambda(0)$ is a Schur measure
\cite{oko}.

Our goal is to describe a perfect sampling algorithm for the pyramidal
Schur process (and thus for the Schur measure). Conceptually speaking,
it consists in sampling an infinite collection of independent
geometric or Bernoulli random variables that has finite support, then
applying our RSK-type correspondence. Let us present this in our
algorithmic setting.
Observe that, when we truncate the weight sequences, i.e.\ when we set
$a_i=b_i=0$ for $i \geq M$, then $\lambda(i)=\emptyset$ for
$|i| \geq M$ and we are back to the finite setting, so that we may
apply the {\tt SchurSample} finite algorithm. Roughly speaking our
generalized algorithm consists in first sampling a suitable $M$ then
applying the finite algorithm modified so as to remove the bias.

More precisely, in the truncated setting, the {\tt SchurSample}
algorithm receives as input the encoded shape $\pi_M= M^M$ (i.e.\ a $M
\times M$ square) and the parameters
$X=(x_1,\ldots,x_M)=(a_{M-1},\ldots,a_0)$ and
$Y=(y_1,\ldots,y_M)=(b_{M-1},\ldots,b_0)$. Recall that generating the
partition $\tau(i,j)$ corresponding to the box $(i,j)$ of $\pi_M$
involves sampling, independently of all the other boxes, a geometric
or Bernoulli random variable with parameter depending on $x_i y_j =
a_{M-i} b_{M-j}$. It is convenient to perform the change of
coordinates $(i,j) \to (M-i,M-j)$ so that to the law of the random
variable $U_{i,j}$ attached to the box $(i,j)$ does not depend on $M$,
as long as $i,j<M$.  Note that, in the new coordinates, we now
generate the $\tau(i,j)$ in decreasing partial order of $(i,j)$,
starting with the boundary conditions
$\tau(M,j)=\tau(i,M)=\emptyset$. The law of $U_{i,j}$ ($0 \leq i,j <
M$) reads explicitly
\begin{equation}
  \label{eq:Udist}
  U_{i,j} \sim
  \begin{cases}
    Geom(c_{i,j}) & \text{where $c_{i,j}=a_i b_j\phantom{/(1+a_ib_j)}$ if $\epsilon_{i,j}=-1$,} \\
    Bernoulli(c_{i,j}) & \text{where $c_{i,j}=a_i b_j/(1+a_ib_j)$ if $\epsilon_{i,j}=1$.}
  \end{cases}
\end{equation}
By a natural coupling, we may realize all values of $M$ on the same
probability space, so that $U_{i,j}$ ($i,j \geq 0$) is a random
variable independent of $M$. Using \eqref{eq:condZ} and the
Borel-Cantelli lemma, it is not difficult to check that, almost
surely, only finitely many $U_{i,j}$ are nonzero. Therefore there
exists $M$ such that $U_{i,j}=0$ unless $0 \leq i,j < M$, and it can
be seen that the output of the truncated algorithm does not depend on
$M$ provided it satisfies this condition.

Rather than sampling $M$ itself, it is convenient to reduce the
problem to a unidimensional one by labeling the boxes by a single
integer. A convenient choice is to label the boxes using the Cantor
pairing function $k(i,j)=(i+j)(i+j+1)/2+j$ which defines a bijection
$\mathbb{N} \times \mathbb{N} \to \mathbb{N}$ ($k(0,0)=0$, $k(1,0)=1$,
$k(0,1)=2$, $k(2,0)=3 \dots$).
Setting $K = \sup \{ k(i,j): i,j \geq 0 \text{ and } \ U_{i,j}>0 \}$,
we have
\begin{equation}
  \label{eq:Kdist}
  \mathbb{P}(K \leq k)= \prod_{\substack{i,j \geq 0\\ k(i,j)>k}} (1 - c_{i,j})
\end{equation}
which is positive by \eqref{eq:condZ} and tends to $1$ as $k \to
\infty$.  This explicit expression makes it theoretically possible to
sample $K$, for instance by drawing a uniform real (Lebesgue) random
variable $V$ on $[0,1]$ and taking the largest $k$ such that
$\mathbb{P}(K \leq k) < V$. Discussing a practical perfect algorithm
to do that on a real-world computer is beyond the scope of this paper,
though this could be conceivably done by computing sufficiently good
bounds for the right hand side of \eqref{eq:Kdist} and sampling enough
digits in the binary expansion of $V$. Let us mention by the way that
the mere efficient sampling of geometric random variables of arbitrary
parameter on a real-world computer is already a nontrivial task
\cite{geomsamp}.

Conditionally on $\{K=k\}$, the distribution of the $U_{i,j}$ is easy
to describe: all $U_{i,j}$ with $k(i,j)>k$ are zero, all $U_{i,j}$
with $k(i,j)<k$ remain distributed as in \eqref{eq:Udist}, and if
$k=k(i,j)$ then $U_{i,j} = 1$ if $\epsilon_{i,j}=1$ and $U_{i,j} \sim
1 + Geom(c_{i,j})$ if $\epsilon_{i,j}=-1$. Once we have generated the
random inputs, we then generate the partitions $\tau(i,j)$ in a way
similar to before. We summarize this discussion in the following
pseudocode and proposition.  \bigskip

\indent {\fs \textbf{Algorithm} {\tt UnboundedSchurSample}} \\
\indent {\fs \textbf{Input:} parameters $a_i$ and $b_i$, $i \geq 0$ \\
\indent initialize a sparse partition-valued matrix $\tau(i,j)=\emptyset$ for all $i,j \geq 0$ \\
\indent sample $K$ distributed as \eqref{eq:Kdist} \\
\indent \textbf{if} $K==-\infty$: \textbf{return} empty Schur process \\
\indent find $(i_0,j_0)$ such that $K=k(i_0,j_0)$, $n_0=i_0+j_0$ \\
\indent \textbf{case} getType($i_0,j_0$): \\
\indent \indent (HV$|$VH): $\tau(i_0,j_0)=1$ \\
\indent \indent HH: $\tau(i_0,j_0)=1+Geom(a_{i_0} b_{i_0})$ \\
\indent \indent VV: $\tau(i_0,j_0)=(1+Geom(a_{i_0} b_{i_0}))'$ \\
\indent \textbf{for} $n=n_0 \dots 1$ \\
\indent\indent \textbf{for} $j=0 \dots n$ \\
\indent\indent\indent \textbf{if} $n==n_0$ and $j \geq j_0$: \textbf{next} \\
\indent\indent\indent $type$ = getType($n_0-j,j$) \\
\indent\indent\indent $\tau(i,j)$=sample($\tau(i+1,j),\tau(i,j+1),\tau(i+1,j+1), a_i b_j, type$) \\
\indent \textbf{Output:} The sequence of partitions $\lambda(i)$, $i \in \mathbb{Z}$, defined by
\begin{align*}
 \lambda(i) := 
 \begin{cases}
  \tau(-i,0) \text{ if $i<0$, }\\
  \tau(0,i) \text{ otherwise.}
 \end{cases}
\end{align*}
}

\begin{prop}
  The output of the algorithm {\tt UnboundedSchurSample} is
  distributed as the pyramidal Schur process of parameters
  $(a_i,b_i)_{i \geq 0}$.
\end{prop}

Our algorithm can be straightforwardly adapted to infinite sequences
of partitions satisfying more general interlacing conditions
than~\eqref{eq:infseq}, and also to free or periodic~\cite{B2007cyl}
boundary conditions. Indeed, in all these models, the partition
function is equal to a convergent infinite product~\cite{bcc}, and
there exists a RSK-type bijection which maps the process to a
countable collection of independent geometric or Bernoulli random
variables whose support is almost surely finite, and which can be
sampled using a similar strategy.

\section{General Schur processes with exponential
  specializations} \label{sec:gensample}

In this section we sketch how to sample the most general Schur process
\cite{or,bor}, defined as a measure over sequences $\Lambda$ of
integer partitions of the form
\begin{equation} \label{eq:genschurseq}
  \emptyset \subset \lambda(1) \supset \mu(1) \subset \lambda(2) \supset \mu(2)
  \subset \cdots \supset \mu(N-1) \subset \lambda(N) \subset \emptyset
\end{equation}
where the probability of a given sequence is given by
\begin{equation} \label{eq:genschurweight}
  Prob(\Lambda) \propto
  s_{\lambda(1)}(\rho_0^+) s_{\lambda(1)/\mu(1)}(\rho_1^-) s_{\lambda(2)/\mu(1)}(\rho_1^+)
  \cdots s_{\lambda(N)/\mu(N-1)}(\rho_{N-1}^+) s_{\lambda(N)}(\rho_N^-).
\end{equation}
Here the parameters $\rho_i^+$ ($i=0,\ldots,N-1$) and $\rho_i^-$
($i=1,\ldots,N$) are nonnegative specializations of the algebra of
symmetric functions, which means that there exist countable
collections of nonnegative real parameters $\alpha^\pm_{i,k}$ and
$\beta^\pm_{i,k}$ ($k \geq 1$), and nonnegative real parameters
$\gamma_i^\pm$, such that
\begin{equation}
  \sum_{k \geq 1}
  (\alpha^\epsilon_{i,k} + \beta^\epsilon_{i,k}) < \infty
  \label{eq:genschurcond}
\end{equation}
for all $i$ and $\epsilon \in \{+,-\}$, and
\begin{equation}
  \label{eq:nonnegspec}
  H(\rho_i^\epsilon;u) := \sum_{n=0}^\infty h_n(\rho_i^\epsilon) u^n =
  e^{\gamma_i^\epsilon u} \prod_{k \geq 1} \frac{1 + \beta^\epsilon_{i,k} u}{
    1 - \alpha^\epsilon_{i,k} u}.
\end{equation}
This indeed suffices to determine the values of the (skew) Schur
functions appearing in \eqref{eq:genschurweight} via
\eqref{eq:jacobitrudi} and \eqref{eq:jacobitrudiskew}.

Definition~\ref{def:schurdef} corresponds to the situation where all
parameters $\gamma_i^\pm$ are zero, and only a finite number of
parameters $\alpha^\pm_{i,k}$ and $\beta^\pm_{i,k}$ are
nonzero. Indeed, assuming that $\alpha^\pm_{i,k}=\beta^\pm_{i,k}$ for
$k>M$, we have
\begin{equation}
  \label{eq:schurfinspec}
  \begin{split}
    s_{\lambda/\mu}(\rho^\epsilon_i) &= \sum_{\nu: \mu \subset \nu
      \subset \lambda}
    s_{\nu/\mu}(\alpha^\epsilon_{i,1},\ldots,\alpha^\epsilon_{i,M})
    s_{\lambda'/\nu'}(\beta^\epsilon_{i,1},\ldots,\beta^\epsilon_{i,M}) \\
    &= \langle \mu | \Gapl(\alpha^\epsilon_{i,1}) \cdots \Gapl(\alpha^\epsilon_{i,M}) \Gatpl(\beta^\epsilon_{i,1}) \cdots \Gatpl(\beta^\epsilon_{i,M})| \lambda \rangle
  \end{split}
\end{equation}
which may be readily expanded as a sum over finite sequences of
interlaced partitions, as in our original definition. (To make the
connection precise we shall take $n=4MN$,
$w=(\prec^M \prec'^M \succ^M \succ'^M)^N$ and get the parameters $Z$
by listing all the parameters $\alpha$ and $\beta$ in a suitable
order. Note that, in Definition~\ref{def:schurdef}, we may assume
without loss of generality $w$ to be of this form, upon taking some
$z_i$'s to be zero hence forcing some pairs of successive partitions
to be equal.)

If we lift the restriction that only a finite number of parameters
$\alpha^\pm_{i,k}$ and $\beta^\pm_{i,k}$ are nonzero, but keep the
restriction that all the $\gamma_i^\pm$ vanish, then we may produce a
perfect sample of the Schur process following the strategy of
Section~\ref{sec:unbounded}, which treats the case $N=1$, but can be
easily adapted to general $N$.

Therefore, to sample the most general Schur process, our only missing
ingredient is an algorithm to handle the $\gamma$ parameters, that
correspond to exponential specializations. For simplicity we again
treat only the case $N=1$ (that is to say we only need to sample a
single partition whose law is a Schur measure) with $\rho_0^+$ a
``pure'' exponential specialization (i.e.\
$\alpha_{0,k}^+=\beta_{0,k}^+=0$ for all $k \geq 1$) and $\rho_1^-$
either another pure exponential specialization (\emph{Plancherel
  case}) or a specialization with $\gamma_1^-=0$ (\emph{mixed
  case}). The adaptation to the most general situation is left to the
reader.

For any fixed skew shape $\lambda/\mu$, we have the well-known
identity \cite{sta}
\begin{equation}
  \label{eq:explim}
  \lim_{n \to \infty} s_{\lambda/\mu}(\underbrace{t/n,\ldots,t/n}_{n \text{ times}}) = 
  f^{\lambda/\mu} \frac{t^{|\lambda/\mu|}}{|\lambda/\mu|!}
\end{equation}
where $f^{\lambda/\mu}$ is the number of standard Young tableaux of
shape $\lambda/\mu$. This corresponds to the so-called exponential
specialization of the ring of symmetric functions. Suppose that, in
the pyramidal Schur process defined in Section~\ref{sec:unbounded},
instead of working with fixed parameters $(a_i,b_i)_{i \geq 0}$, we
take
\begin{equation}
  \label{eq:expspec}
  a_i^{(n)} =
  \begin{cases}
    \frac{a}{n} & \text{if $0 \leq i \leq n-1$} \\
    0 & \text{if $i \geq n$}
  \end{cases}
\end{equation}
and we then let $n \to \infty$. We may either take the $b$'s of a
similar form (Plancherel case), or keep them fixed (mixed case). Then
we readily see that the law of $\lambda(0)$ converges weakly as $n \to
\infty$ to the probability measure given by
\begin{equation}
  \label{eq:expconv}
  Prob(\lambda(0)) \propto
  \begin{cases}
    \displaystyle (ab)^{|\lambda|} \left(\frac{f^{\lambda}}{|\lambda|!}\right)^2 & \text{(Plancherel case)}, \\
    \displaystyle a^{|\lambda|} \frac{f^{\lambda}}{|\lambda|!} s_\lambda(b_0,b_1,\ldots) & \text{(mixed case)}. \\
  \end{cases}
\end{equation}
The Plancherel case indeed corresponds to the (poissonized) Plancherel
measure on integer partitions, see~\cite{oko} and references
therein. It can also be seen that the full process converges, up to a
suitable rescaling of the ``time'' (for instance the law of
$\lambda(-\lfloor xn \rfloor)$ converges for any $x \geq 0$).

In the Plancherel case, it is well-known that a perfect sampling
algorithm is given by a poissonized version of the classical
Robinson--Schensted correspondence, or Fomin's equivalent description
in terms of growth diagrams, see e.g.\ \cite[Appendix]{sta} or
\cite{kra}. Indeed, applying our {\tt SchurSample} algorithm to the
finite approximation of order $n$ \eqref{eq:expspec}, the encoded
shape is a rectangle of width $n$ and we need to sample independently
for each box $(i,j)$ a geometric or Bernoulli random variable whose
parameter is (essentially) $(ab)/n^2$ (Plancherel case) or $a b_j/n$
(mixed case). These samples form an array of nonnegative integers
whose nonzero entries converge, in the large $n$ limit and after a
suitable rescaling, to a simple Poisson point process, which is
two-dimensional in the Plancherel case and unidimensional (one per
``$b_i$ line'') in the mixed case. However, we need not sample such a
Poisson point process fully, since only the relative order of the
coordinates of the points matters. More precisely, it is easily seen
that, in the finite approximation of order $n$, we have
$\tau(i+1,j)=\tau(i,j)$ (resp.\ $\tau(i,j)=\tau(i,j+1)$) whenever the
array contains no nonzero entry in row $i$ (resp.\ in column
$j$). Thus, in practice, $\tau(i,j)$ takes only a finite (but random)
number of values in the large $n$ limit. As mentioned above this
reduces in the Plancherel case to Fomin's description of the RS
correspondence, as the relevant data from the 2D Poisson point process
is nothing but a random permutation of random length $Poisson(ab)$. In the
mixed case, we obtain a different algorithm which interpolates
between the classical RS and the RSK algorithms. The corresponding measure was
studied by O'Connell \cite{oco} from the perspective of conditioned random walks. 
It corresponds to a Poissonized version of the RS bijection (in the terminology of 
\cite{oco}, Section 2.2) between words $\{1,\dots,k\} \to \mathbb{N}^*$ and pairs
$(P,Q)$ of tableaux with $P$ semi-standard and $Q$ standard. In the measure, $Q$ gives
rise to the Plancherel factor and $P$ to the Schur function. We also note that the 
dynamics (sampling algorithm) in the mixed case, once one takes $n \to \infty$,
becomes a continuous time dynamics governed by exponential clocks (one per ``$b_i$ line'')
which appeared in the work of Borodin and Ferrari \cite{bf}.

\section{Conclusion} \label{sec:conclusion}

In this article we have presented an efficient polynomial time algorithm for sampling from Schur processes and symmetric Schur processes. It makes minimal use of randomness and leads to efficient sampling algorithms for a variety of tilings including plane partitions, the Aztec diamond and pyramid partitions. 

The Schur functions have generalizations in the Macdonald functions, and while an algorithm can be written down in this case, it will not be polynomial for the branching (Pieri) coefficients in this general setting depend on the skew diagram. It is unlikely that an efficient sampling algorithm will exist (but see \cite{bp} for dynamics generalizing RSK to the Macdonald level). However, Macdonald measures on steep tilings are of interest independent of sampling, and we hope to address this in the future. Also of interest is the case where one replaces Schur by Schur's $P$-functions (a degenerate case of the Macdonald hierarchy). Here, because of the ``free-fermionic'' nature of the problem, there may be hope for a fast perfect sampling algorithm.

In a different direction, random sampling can be used to conjecture and prove various laws of large numbers for Schur processes. Some such laws have been proven, but more await discovery, especially regarding steep tilings, as we illustrated on Figure~\ref{fig:pyramidCusp}. This will be addressed in future work.

\section*{Acknowledgments}

We thank Alexei Borodin, Patrik Ferrari and Philippe Biane for useful discussions. The present manuscript was finalized while several of the authors were visiting the Galileo Galilei Institute in Firenze, and we acknowledge the
hospitality extended to us under the program ``Statistical Mechanics, Integrability and Combinatorics''. The authors further acknowledge financial support from the Agence Nationale de la
Recherche via the grants ANR-08-JCJC-0011 ``IComb'' (SC),
ANR-10-BLAN-0123 ``MAC2'' (DB, CB), ANR-12-JS02-0001 ``Cartaplus''
(JB, GC), ANR-14-CE25-0014 ``GRAAL'' (JB) and from the Ville de Paris
via Projet \'Emergences ``Combinatoire \`a Paris'' (JB, GC, SC).

\bibliographystyle{plain}
\bibliography{samplingpaper}

\end{document}